\documentclass[12pt]{article}
\pdfoutput=1
\usepackage{amssymb}
\usepackage{amsthm}
\usepackage{amsmath}
\usepackage{latexsym}
\usepackage[utf8]{inputenc}
\usepackage[russian, english]{babel}
\usepackage{graphicx}
\usepackage[justification=centering, labelfont=bf]{caption}
\usepackage{subcaption}
\usepackage{mathtools}
\usepackage[hidelinks]{hyperref}
\usepackage[inline]{enumitem}
\usepackage{mathrsfs}
\usepackage{thmtools}
\usepackage{stmaryrd}
\usepackage{eucal}

\usepackage{tikz}
\usetikzlibrary{shapes,snakes}
\usetikzlibrary{arrows.meta}

\usepackage[backend=biber, style=numeric, sorting=nyt, maxnames=100]{biblatex}
\usepackage[left=2.4cm,right=2.4cm,top=2.4cm,bottom=2.4cm,bindingoffset=0cm]{geometry}

\title{Sharp Dirac's Theorem for DP-Critical Graphs}
\date{}

\author{
	Anton~Bernshteyn\thanks{Department of Mathematics, University of Illinois at Urbana--Champaign, IL, USA. Email address: \href{mailto:bernsht2@illinois.edu}{\texttt{bernsht2@illinois.edu}}. Research of this author is supported by the Illinois Distinguished Fellowship.}
	\and Alexandr Kostochka\thanks{Department of Mathematics, University of Illinois at Urbana--Champaign, IL, USA and
		Sobolev Institute of Mathematics, Novosibirsk 630090, Russia. Email address: \href{mailto:kostochk@math.uiuc.edu}{\texttt{kostochk@math.uiuc.edu}}. Research of this author is supported in part by NSF grant
		DMS-1600592 and grants 15-01-05867 and 16-01-00499 of the Russian Foundation for Basic Research.}}

\newtheorem{theo}{Theorem}[section]
\newtheorem{prop}[theo]{Proposition}
\newtheorem{lemma}[theo]{Lemma}
\newtheorem{corl}[theo]{Corollary}

\newtheorem{problem}[theo]{Problem}
\newtheorem{claim}{Claim}[theo]

\newcommand*{\myproofname}{Proof}
\newenvironment{claimproof}[1][\myproofname]{\begin{proof}[#1]}{\end{proof}}

\theoremstyle{definition}
\newtheorem{defn}[theo]{Definition}
\newtheorem{exmp}[theo]{Example}

\newtheorem{remk}[theo]{Remark}

\newcommand{\0}{\varnothing}
\newcommand{\set}[1]{\{#1\}}
\newcommand{\dom}{\operatorname{dom}}
\newcommand{\N}{\mathbb{N}}

\renewcommand{\c}[1]{{#1}^c}
\renewcommand{\epsilon}{\varepsilon}
\renewcommand{\phi}{\varphi}

\newcommand{\powerset}[1]{\operatorname{Pow}(#1)}
\newcommand{\defeq}{\coloneqq}
\newcommand{\I}{\mathbf{Ind}}
\newcommand{\D}{\mathbf{Dir}}
\newcommand{\Cov}[1]{\mathscr{#1}}

\numberwithin{equation}{section}

\renewbibmacro{in:}{}

\renewbibmacro*{volume+number+eid}{%
	\printfield{volume}%
	\setunit*{\addnbspace}
	\printfield{number}%
	\setunit{\addcomma\space}%
	\printfield{eid}}

\DeclareFieldFormat[article]{volume}{\textbf{#1}\space}
\DeclareFieldFormat[article]{number}{\mkbibparens{#1}}

\DeclareFieldFormat{journaltitle}{#1,}
\DeclareFieldFormat[thesis]{title}{\mkbibemph{#1}\addperiod}
\DeclareFieldFormat[article, unpublished, thesis]{title}{\mkbibemph{#1},}
\DeclareFieldFormat[book]{title}{\mkbibemph{#1}\addperiod}
\DeclareFieldFormat[unpublished]{howpublished}{#1, }

\DeclareFieldFormat{pages}{#1}

\DeclareFieldFormat[article]{series}{Ser.~#1\addcomma}

\begin{document}

	\maketitle
	
	\begin{abstract}
		\noindent \emph{Correspondence coloring}, or \emph{DP-coloring}, is a generalization of list coloring introduced recently by Dvo\v r\' ak and Postle~\cite{DP}. 
		In this paper we establish a version of Dirac's theorem on the minimum number of edges in critical graphs~\cite{Dirac1} in the framework of DP-colorings. 
		A corollary of our main result answers a question posed by Kostochka and Stiebitz~\cite{KostochkaStiebitz} on classifying list-critical graphs that satisfy Dirac's bound with equality.
	\end{abstract}

	\section{Introduction}
	
	All graphs considered here are finite, undirected, and simple. We use $\N$ to denote the set of all nonnegative integers. For $k \in \N$, let $[k] \defeq \set{1\ldots, k}$. For a set $S$, we use $\powerset{S}$ to denote the power set of $S$, i.e., the set of all subsets of $S$. For a function $f \colon A \to B$ and a subset $S \subseteq A$, we use $f \vert_S$ to denote the restriction of $f$ to $S$. For a graph $G$, $V(G)$ and $E(G)$ denote the vertex and the edge sets of $G$, respectively. For a set $U \subseteq V(G)$, $G[U]$ is the subgraph of $G$ induced by $U$. Let $G - U \defeq G[V(G) \setminus U]$, and for $u \in V(G)$, let $G - u \defeq G - \set{u}$. For two subsets $U_1$, $U_2 \subseteq V(G)$, $E_G(U_1, U_2) \subseteq E(G)$ denotes the set of all edges in $G$ with one endpoint in $U_1$ and the other one in $U_2$. For $u \in V(G)$, $N_G(u)\subset V(G)$ denotes the set of all neighbors of $u$ and $\deg_G(u) \defeq |N_G(u)|$ denotes the degree of $u$ in $G$. We let $\Delta(G) \defeq \max_{u \in V(G)} \deg_G(u)$ and $\delta(G) \defeq \min_{u \in V(G)} \deg_G(u)$ denote the maximum and the minimum degrees of $G$, respectively.
	For a subset $U \subseteq V(G)$, let $N_G(U) \defeq \bigcup_{u \in U} N_G(u)$ denote the neighborhood of $U$ in $G$. A~set $I \subseteq V(G)$ is {\em independent} if $I \cap N_G(I) = \0$, i.e., if $uv \not \in E(G)$ for all $u$, $v \in I$. We denote the family of all independent sets in a graph $G$ by $\I(G)$. The complete graph on $n$ vertices is denoted by~$K_n$.
	
	\subsection{Critical graphs and the theorems of Brooks, Dirac, and Gallai}
	
	Recall that a \emph{proper coloring} of a graph $G$ is a function $f \colon V(G) \to Y$, where $Y$ is a set, whose elements are referred to as \emph{colors}, 
	such that  $f(u) \neq f(v)$ for each edge $uv \in E(G)$. The smallest $k \in \N$ such that there exists a proper coloring $f \colon V(G) \to Y$ with $|Y| = k$ is called the \emph{chromatic number} of $G$ and is denoted by $\chi(G)$.
	
	For $k \in \N$, a graph $G$ is said to be \emph{$(k+1)$-vertex-critical} if $\chi(G) = k+1$ but $\chi(G-u) \leq k$ for all $u \in V(G)$. We will only consider vertex-critical graphs in this paper, so for brevity we will call them simply \emph{critical}. Since every graph~$G$ with $\chi(G) > k$ contains a $(k+1)$-critical subgraph, understanding the structure of critical graphs is crucial for the study of graph coloring. We will only consider $k \geq 3$, the case $k \leq 2$ being trivial (the only $1$-critical graph is $K_1$, the only $2$-critical graph is $K_2$, and the only $3$-critical graphs are odd cycles).
	
	Let $k \geq 3$ and suppose that $G$ is a $(k+1)$-critical graph with $n$ vertices and $m$ edges. A classical problem in the study of critical graphs is to understand how small $m$ can be depending on~$n$ and $k$.  Evidently, $\delta(G) \geq k$; in particular, $2m \geq kn$. Brooks's Theorem is equivalent to the assertion that the only situation in which $2m = kn$ is when $G \cong K_{k+1}$:
	
	\begin{theo}[{Brooks~\cite[Theorem~14.4]{BondyMurty}}]
		Let $k \geq 3$ and let $G$ be a $(k+1)$-critical graph distinct from $K_{k+1}$. Set $n \defeq |V(G)|$ and $m \defeq |E(G)|$. Then
		$$
		2m > kn.
		$$
	\end{theo}
	
	Brooks's theorem was subsequently sharpened by Dirac, who established a linear in $k$ lower bound on the difference $2m - kn$:
	
	\begin{theo}[{Dirac~\cite[Theorem~15]{Dirac1}}]\label{theo:Dirac}
		Let $k \geq 3$ and let $G$ be a $(k+1)$-critical graph distinct from $K_{k+1}$. Set $n \defeq |V(G)|$ and $m \defeq |E(G)|$. Then
		\begin{equation}\label{eq:Dirac}
		2m \geq kn + k-2.
		\end{equation}
	\end{theo}
	
	Bound~\eqref{eq:Dirac} is sharp in the sense that for every $k\geq 3$, there exist $(k+1)$-critical graphs that satisfy $2m = kn + k - 2$. However, for each $k$, there are only finitely many such graphs; in fact, they admit a simple characterization, which we present below.
	
	\begin{defn}
		Let $k \geq 3$. A graph $G$ is \emph{$k$-Dirac} if its vertex set can be partitioned into three subsets $V_1$, $V_2$, $V_3$ such that:
		\begin{itemize}
			\item[--] $|V_1| = k$, $|V_2| = k-1$, and $|V_3| = 2$;
			\item[--] the graphs $G[V_1]$ and $G[V_2]$ are complete;
			\item[--] each vertex in $V_1$ is adjacent to exactly one vertex in $V_3$;
			\item[--] each vertex in $V_3$ is adjacent to at least one vertex in $V_1$;
			\item[--] each vertex in $V_2$ is adjacent to both vertices in $V_3$; and
			\item[--] $G$ has no other edges.
		\end{itemize}
		We denote the family of all $k$-Dirac graphs by~$\D_k$.
	\end{defn}
	
	\begin{theo}[{Dirac~\cite[Theorem, p.~152]{Dirac2}}]\label{theo:sharp_Dirac}
		Let $k \geq 3$ and let $G$ be a $(k+1)$-critical graph distinct from $K_{k+1}$. Set $n \defeq |V(G)|$ and $m \defeq |E(G)|$. Then
		$$
		2m = kn + k-2\,\Longleftrightarrow\, G \in \D_k.
		$$
	\end{theo}
	
	As $n$ goes to infinity, the gap between Dirac's lower bound and the sharp bound increases. In fact, Gallai~\cite{Gallai} observed that the asymptotic density of large $(k+1)$-critical graphs distinct from $K_{k+1}$ is strictly greater than $k/2$. However, Gallai's bound is stronger than~\eqref{eq:Dirac} only for $n$ at least quadratic in $k$.
	
	\subsection{List coloring}
	
	\emph{List coloring} was introduced independently by Vizing~\cite{Vizing} and Erd\H os, Rubin, and Taylor~\cite{ERT}. A \emph{list assignment} for a graph $G$ is a function $L \colon V(G) \to \powerset{Y}$, where $Y$ is a set, whose elements, as in the case of ordinary colorings, are referred to as \emph{colors}. For each $u \in V(G)$, the set $L(u)$ is called the \emph{list} of $u$ and its elements are said to be \emph{available} for~$u$. If $|L(u)| = k$ for all $u \in V(G)$, then~$L$ is called a \emph{$k$-list assignment}. A proper coloring $f \colon V(G) \to Y$ is called an \emph{$L$-coloring} if  $f(u) \in L(u)$ for each $u \in V(G)$. A graph~$G$ is said to be \emph{$L$-colorable} if it has an $L$-coloring. The \emph{list-chromatic number} $\chi_\ell(G)$ of $G$  is the smallest $k \in \N$ such that $G$ is $L$-colorable for every $k$-list assignment $L$ for $G$. If $k \in \N$ and $L(u) = [k]$ for all $u \in V(G)$, then $G$ is $L$-colorable if and only if it is $k$-colorable; in this sense, list coloring generalizes
	ordinary coloring. In particular, $\chi_\ell(G) \geq \chi(G)$ for all graphs $G$.
	
	A list assignment $L$ for a graph $G$ is called a \emph{degree list assignment} if $|L(u)| \geq \deg_G(u)$ for all $u \in V(G)$. A fundamental result of Borodin~\cite{Borodin} and Erd\H os--Rubin--Taylor~\cite{ERT}, which can be seen as a generalization of Brooks's theorem to list colorings, provides a complete characterization of all graphs $G$ that are not $L$-colorable with respect to some degree list assignment $L$.
	
	\begin{defn}
		A \emph{Gallai tree} is a connected graph in which every block is either a clique or an odd cycle. A~\emph{Gallai forest} is a graph in which every connected component is a Gallai tree.
	\end{defn}
	
	\begin{theo}[{Borodin~\cite{Borodin}; Erd\H os--Rubin--Taylor~\cite[Theorem, p.~142]{ERT}}]\label{theo:list_Brooks}
		Let $G$ be a connected graph and let $L$ be a degree list assignment for $G$. If $G$ is not $L$-colorable, then $G$ is a Gallai tree; furthermore, $|L(u)| = \deg_G(u)$ for all $u \in V(G)$, and if $u$, $v \in V(G)$ are two adjacent non-cut vertices, then $L(u) = L(v)$.
	\end{theo}
	
	Theorem~\ref{theo:list_Brooks} provides some useful information about the structure of critical graphs:
	
	\begin{corl}\label{corl:D}
		Let $k \geq 3$ and let $G$ be a $(k+1)$-critical graph. Set
		$$
		D \defeq \set{u \in V(G) \,:\, \deg_G(u) = k}.
		$$
		Then $G[D]$ is a Gallai forest.
	\end{corl}
	
	Corollary~\ref{corl:D} was originally proved by Gallai~\cite{Gallai} using a different method. It is crucial for the proof of Gallai's theorem on the asymptotic average degree of $(k+1)$-critical graphs.
	
	The definition of critical graphs can be naturally extended to list colorings. A~graph $G$ is said to be \emph{$L$-critical}, where $L$ is a list assignment for $G$, if $G$ is not $L$-colorable but for any $u \in V(G)$, the graph $G-u$ is $L\vert_{V(G - u)}$-colorable. Note that if we set $L(u) \defeq [k]$ for all $u \in V(G)$, then $G$ being $L$-critical is equivalent to it being $(k+1)$-critical. Repeating the argument used to prove Corollary~\ref{corl:D}, we obtain the following more general statement:
	
	\begin{corl}[{Kostochka--Stiebitz--Wirth~\cite[Theorem~5]{KSW}}]\label{corl:D1}
		Let $k \geq 3$ and let $G$ be a graph. Suppose that $L$ is a $k$-list assignment for $G$ such that $G$ is $L$-critical. Set
		$$
		D \defeq \set{u \in V(G) \,:\, \deg_G(u) = k}.
		$$
		Then $G[D]$ is a Gallai forest.
	\end{corl}
	
	Corollary~\ref{corl:D1} can be used to prove a version of Gallai's theorem for list-critical graphs, i.e., to show that the average degree of a graph $G$ distinct from $K_{k+1}$ that is $L$-critical for some $k$-list assignment $L$ has average degree strictly greater than~$k/2$. On the other hand, list-critical graphs distinct from $K_{k+1}$ do not, in general, admit a nontrivial lower bound on the difference $2m-kn$ that only depends on $k$ (analogous to the one given by Dirac's Theorem~\ref{theo:Dirac} for $(k+1)$-critical graphs). Consider the following example, presented in~\cite[p.~167]{KostochkaStiebitz}. Fix $k \in \N$ and let~$G$ be the graph with vertex set $\set{a_0, \ldots, a_k, b_0, \ldots, b_k}$ of size $2(k+1)$ and edge set $\set{a_i a_j, b_i b_j \,:\, i \neq j} \cup \set{a_0 b_0}$. For each $i \in [k]$, let $L(a_i) = L(b_i) \defeq [k]$, and let $L(a_0) = L(b_0) \defeq \set{0} \cup [k-1]$. Then~$G$ is $L$-critical; however, $2|E(G)|
	- k|V(G)| = 2$.
	
	Nonetheless, Theorem~\ref{theo:Dirac} can  be extended to the list coloring framework if we restrict our attention to graphs that do not contain $K_{k+1}$ as a \emph{subgraph}:
	
	\begin{theo}[{Kostochka--Stiebitz~\cite[Theorem~2]{KostochkaStiebitz}}]\label{theo:list_Dirac}
		Let $k \geq 3$. Let $G$ be a graph and let $L$ be a $k$-list assignment for~$G$ such that $G$ is $L$-critical. Suppose that $G$ does not contain a clique of size $k+1$. Set $n \defeq |V(G)|$ and $m \defeq |E(G)|$. Then
		$$
		2m \geq kn + k-2.
		$$
	\end{theo}
	
	Kostochka and Stiebitz~\cite[Section~4]{KostochkaStiebitz} 
asked whether the conclusion of Theorem~\ref{theo:sharp_Dirac} also holds for list critical graphs with no $K_{k+1}$ as a subgraph. We answer this question in the affirmative; see~Corollary~\ref{corl:sharp_list_Dirac}.
	
	\subsection{DP-colorings and the results of this paper}
	
	In this paper we focus on a generalization of list coloring that was recently introduced by Dvo\v r\' ak and Postle~\cite{DP}; they called it \emph{correspondence coloring}, and we call it \emph{DP-coloring} for short. Dvo\v r\' ak and Postle invented DP-coloring in order to prove that every planar graph without cycles of lengths $4$ to $8$ is $3$-list-colorable~\cite[Theorem~1]{DP}, thus answering a long-standing question of Borodin~\cite[Problem~8.1]{Bor13}. In the setting of DP-coloring, not only does each vertex get its own list of available colors, but also the identifications between the colors in the lists can vary from edge to edge.
	
	\begin{defn}\label{defn:cover}
		Let $G$ be a graph. A \emph{cover} of $G$ is a pair $\Cov{H} = (L, H)$, consisting of a graph $H$ and a function $L \colon V(G) \to \powerset{V(H)}$, satisfying the following requirements:
		\begin{enumerate}[labelindent=\parindent,leftmargin=*,label=(C\arabic*)]
			\item the sets $\set{L(u) \,:\,u \in V(G)}$, form a partition of $V(H)$;
			\item for every $u \in V(G)$, the graph $H[L(u)]$ is complete;
			\item if $E_H(L(u), L(v)) \neq \0$, then either $u = v$ or $uv \in E(G)$;
			\item \label{item:matching} if $uv \in E(G)$, then $E_H(L(u), L(v))$ is a matching.
		\end{enumerate}
		A cover $\Cov{H} = (L, H)$ of $G$ is \emph{$k$-fold} if $|L(u)| = k$ for all $u \in V(G)$.
	\end{defn}
	
	\begin{remk}
		The matching $E_H(L(u), L(v))$ in Definition~\ref{defn:cover}\ref{item:matching} does not have to be perfect and, in particular, is allowed to be empty.
	\end{remk}
	
	
	
	\begin{defn}
		Let $G$ be a graph and let $\Cov{H} = (L, H)$ be a cover of $G$. An \emph{$\Cov{H}$-coloring} of $G$ is an independent set $I \in \I(H)$ of size $|V(G)|$. We say that $G$ is \emph{$\Cov{H}$-colorable} if it has an $\Cov{H}$-coloring.
	\end{defn}
	
	\begin{remk}\label{remk:single}
		By definition, if $\Cov{H} = (L, H)$ is a cover of $G$, then $\set{L(u)\,:\, u \in V(G)}$ is a partition of~$H$ into $|V(G)|$ cliques. Therefore, $I \in \I(H)$ is an $\Cov{H}$-coloring of $G$ if and only if $|I \cap L(u)| = 1$ for all $u \in V(G)$.
	\end{remk}
	
	\begin{defn}
		Let $G$ be a graph. The \emph{DP-chromatic number} $\chi_{DP}(G)$ of $G$ is the smallest $k \in \N$ such that $G$ is $\Cov{H}$-colorable for every $k$-fold cover $\Cov{H}$ of $G$.
	\end{defn}
	
	\begin{exmp}\label{exmp:cycles}
		Figure~\ref{fig:cycle} shows two distinct $2$-fold covers of the $4$-cycle $C_4$. Note that $C_4$ is $\Cov{H}_1$-colorable but not $\Cov{H}_2$-colorable. In particular, $\chi_{DP}(C_4) \geq 3$. On the other hand, it can be easily seen that $\chi_{DP}(G) \leq \Delta(G) + 1$ for any graph $G$, and so we have $\chi_{DP}(C_4) = 3$. A~similar argument demonstrates that $\chi_{DP}(C_n) = 3$ for any cycle $C_n$ of length $n \geq 3$.
	\end{exmp}
	
	\begin{figure}[h]
		\centering	
		\begin{tikzpicture}[scale=0.7]
		\definecolor{light-gray}{gray}{0.95}
		
		\filldraw[fill=light-gray]
		(6.5,0) circle [x radius=1cm, y radius=5mm, rotate=45]
		(6.5,3) circle [x radius=1cm, y radius=5mm, rotate=-45]
		(9.5,0) circle [x radius=1cm, y radius=5mm, rotate=-45]
		(9.5,3) circle [x radius=1cm, y radius=5mm, rotate=45];
		
		\foreach \x in {(6, -0.5), (6, 3.5), (7, 0.5), (7, 2.5), (9, 0.5), (9, 2.5), (10, -0.5), (10, 3.5)}
		\filldraw \x circle (4pt);
		
		\draw[thick] (6, -0.5) -- (6, 3.5) -- (10, 3.5) -- (10, -0.5) -- cycle;
		
		\draw[thick] (7, 0.5) -- (7, 2.5) -- (9, 2.5) -- (9, 0.5) -- cycle;
		
		\draw[thick] (6, -0.5) -- (7, 0.5) (6, 3.5) -- (7, 2.5) (9, 2.5) -- (10, 3.5) (10, -0.5) -- (9, 0.5);
		
		\node at (8, -1.5) {$\Cov{H}_1$};
		
		\filldraw[fill=light-gray]
		(13+3,0) circle [x radius=1cm, y radius=5mm, rotate=45]
		(13+3,3) circle [x radius=1cm, y radius=5mm, rotate=-45]
		(16+3,0) circle [x radius=1cm, y radius=5mm, rotate=-45]
		(16+3,3) circle [x radius=1cm, y radius=5mm, rotate=45];
		
		\foreach \x in {(12.5+3, -0.5), (12.5+3, 3.5), (13.5+3, 0.5), (13.5+3, 2.5), (15.5+3, 0.5), (15.5+3, 2.5), (16.5+3, -0.5), (16.5+3, 3.5)}
		\filldraw \x circle (4pt);
		
		\draw[thick] (12.5+3, -0.5) -- (12.5+3, 3.5) -- (16.5+3, 3.5) -- (16.5+3, -0.5) -- (13.5+3, 0.5) -- (13.5+3, 2.5) -- (15.5+3, 2.5) -- (15.5+3, 0.5) -- cycle;
		
		\draw[thick] (12.5+3, -0.5) -- (13.5+3, 0.5) (12.5+3, 3.5) -- (13.5+3, 2.5) (15.5+3, 0.5) -- (16.5+3, -0.5) (15.5+3, 2.5) -- (16.5+3, 3.5);
		
		\node at (14.5+3, -1.5) {$\Cov{H}_2$};
		\end{tikzpicture}
		\caption{Two distinct $2$-fold covers of a $4$-cycle.}\label{fig:cycle}
	\end{figure}
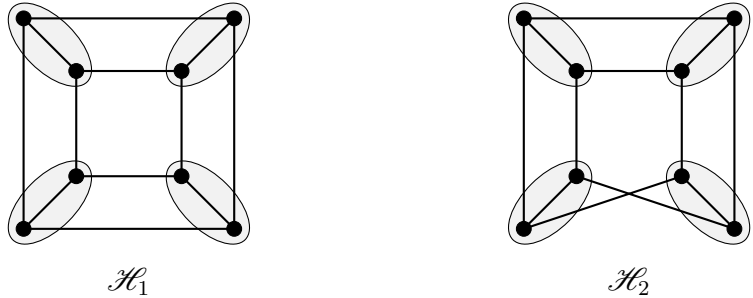
	
	In order to see that DP-colorings indeed generalize list colorings, consider a graph $G$ and a list assignment $L$ for $G$. Let $H$ be the graph with vertex set
	$$
	V(H) \defeq \set{(u, c)\,:\, u \in V(G) \text{ and } c \in L(u)},
	$$
	in which two distinct vertices $(u, c)$ and $(v, d)$ are adjacent if and only if
	\begin{itemize}
		\item[--] either $u = v$,
		\item[--] or else, $uv \in E(G)$ and $c = d$.
	\end{itemize}
	For each $u \in V(G)$, set
	$$
	L'(u) \defeq \set{(u, c) \,:\, c \in L(u)}.
	$$
	Then $\Cov{H} \defeq (L', H)$ is a cover of $G$, and there is a natural bijective correspondence between the $L$-colorings and the $\Cov{H}$-colorings of $G$. Indeed, if $f$ is an $L$-coloring of $G$, then the set
	$$
	I_f \defeq \set{(u, f(u)) \,:\, u \in V(G)}
	$$
	is an $\Cov{H}$-coloring of $G$. Conversely, given an $\Cov{H}$-coloring $I \in \I(H)$ of $G$, $|I \cap L'(u)| = 1$ for all $u \in V(G)$, so one can define an $L$-coloring $f_I$ by the property $$(u, f_I(u)) \in I \cap L'(u),$$ for all $u \in V(G)$. This shows that list colorings can be identified with a subclass of DP-colorings. In particular, $\chi_{DP}(G) \geq \chi_\ell(G)$ for all graphs $G$.
	
	Some upper bounds on list-chromatic numbers hold for DP-chromatic numbers as well. For instance, it is easy to see that $\chi_{DP}(G) \leq d+1$ for any $d$-degenerate graph $G$. Dvo\v r\'ak and Postle~\cite{DP} pointed out that for any planar graph $G$, $\chi_{DP}(G) \leq 5$ and, moreover, $\chi_{DP}(G) \leq 3$ if $G$ is a planar graph of girth at least $5$ (these statements are extensions of classical results of Thomassen~\cite{Thomassen1, Thomassen2} regarding list colorings). On the other hand, there are also some striking differences between DP- and list colorings. For example, even cycles are $2$-list-colorable, while their DP-chromatic number is $3$ (see Example~\ref{exmp:cycles}). In particular, the orientation theorems of Alon--Tarsi~\cite{AT} and the Bondy--Boppana--Siegel lemma (see~\cite{AT}) do not extend to DP-colorings; see~\cite{BKExamples} for more examples demonstrating the failure of these techniques in the DP-coloring context. Bernshteyn~\cite[Theorem~1.6]{Bernshteyn} showed that the DP-chromatic number of every graph with average degree $d$ is $\Omega(d/\log d)$, i.e., almost linear in $d$ (recall that due to a celebrated result of Alon~\cite{Alon}, the list-chromatic number of such graphs is $\Omega(\log d)$, and this bound is best possible). On the other hand, Johansson's upper bound~\cite{Johansson} on list chromatic numbers of triangle-free graphs also holds for DP-chromatic numbers~\cite[Theorem~1.7]{Bernshteyn}.
	
	A cover $\Cov{H} = (L, H)$ of a graph $G$ is a \emph{degree cover} if $|L(u)| \geq \deg_G(u)$ for all $u \in V(G)$. Bernshteyn, Kostochka, and Pron~\cite{BKP} established the following generalization of Theorem~\ref{theo:list_Brooks}:
	
	\begin{defn}
		A \emph{GDP-tree} is a connected graph in which every block is either a clique or a cycle. A \emph{GDP-forest} is a graph in which every connected component is a GDP-tree.
	\end{defn}
	
	\begin{theo}[{\cite[Theorem~9]{BKP}}]\label{theo:DP_Brooks}
		Let $G$ be a connected graph and let $\Cov{H} = (L, H)$ be a degree cover of $G$. If $G$ is not $\Cov{H}$-colorable, then $G$ is a GDP-tree; furthermore, $|L(u)| = \deg_G(u)$ for all $u \in V(G)$, and if $u$, $v \in V(G)$ are two adjacent non-cut vertices, then $E_H(L(u), L(v))$ is a perfect matching.
	\end{theo}
	
	Let $G$ be a graph and let $\Cov{H} = (L,H)$ be a cover of $G$. We say that $G$ is \emph{$\Cov{H}$-critical} if $G$ is not $\Cov{H}$-colorable but for any $u \in V(G)$, there exists $I \in \I(H)$ such that $I \cap L(v) \neq \0$ for all $v \neq u$. Theorem~\ref{theo:DP_Brooks} implies the following:
	
	\begin{corl}[\cite{BKP}]\label{corl:D2}
		Let $k \geq 3$ and let $G$ be a graph. Suppose that $\Cov{H}$ is a $k$-fold cover of $G$ such that $G$ is $\Cov{H}$-critical. Set
		$$
		D \defeq \set{u \in V(G) \,:\, \deg_G(u) = k}.
		$$
		Then $G[D]$ is a GDP-forest.
	\end{corl}
	
	Corollary~\ref{corl:D2} implies an extension of Gallai's theorem to DP-critical graphs~\cite[Corollary~10]{BKP}.
	
	The main result of this paper is a generalization of Theorem~\ref{theo:list_Dirac} to DP-critical graphs. In fact, we establish a sharp version that also generalizes Theorem~\ref{theo:sharp_Dirac}:
	
	\begin{theo}\label{theo:sharp_DP_Dirac} Let $k\geq 3$.
		Let $G$ be a graph and let $\Cov{H}$ be a $k$-fold cover of $G$ such that $G$ is $\Cov{H}$-critical. Suppose that $G$ does not contain a clique of size $k+1$. Set $n \defeq |V(G)|$ and $m \defeq |E(G)|$. If $G \not \in \D_k$, then
		$$
		2m > kn + k-2.
		$$
	\end{theo}
	
	An immediate corollary of Theorem~\ref{theo:sharp_DP_Dirac} is the following version of Theorem~\ref{theo:sharp_Dirac} for list colorings:
	
	\begin{corl}\label{corl:sharp_list_Dirac}  Let $k\geq 3$.
		Let $G$ be a graph and let $L$ be a $k$-list assignment for~$G$ such that~$G$ is $L$-critical. Suppose that $G$ does not contain a clique of size $k+1$. Set $n \defeq |V(G)|$ and $m \defeq |E(G)|$. If $G \not \in \D_k$, then
		$$
		2m > kn + k-2.
		$$
	\end{corl}
	
	Our proof of Theorem~\ref{theo:sharp_DP_Dirac} is, essentially, inductive. As often is the case, having a stronger inductive assumption (due to considering DP-critical and not just list-critical graphs) allows for more flexibility in the proof. In particular, we do not know if our argument can be adapted to give a ``DP-free'' proof of Corollary~\ref{corl:sharp_list_Dirac}.
	
	\section{Proof of Theorem~\ref{theo:sharp_DP_Dirac}: first observations}
	
	\subsection{Set-up and notation}
	
	From now on, we fix a counterexample to Theorem~\ref{theo:sharp_DP_Dirac}; more precisely, we fix the following data:
	\begin{itemize}
		\item an integer $k \geq 3$;
		\item a graph $G$ with $n$ vertices and $m$ edges such that:
		\begin{itemize}
				\item $G \not \in \D_k$;
				\item $G$ does not contain a clique of size $k+1$; and
				\item $G$ satisfies the inequality
				\begin{equation}\label{eq:main_inequality}
					2m \leq kn+ k-2;
				\end{equation}
		\end{itemize}
		\item a $k$-fold cover $\Cov{H} = (L, H)$ of $G$ such that $G$ is $\Cov{H}$-critical.
	\end{itemize}
	Furthermore, we assume that $G$ is a counterexample with the fewest vertices.
	
	For brevity, we denote $V \defeq V(G)$ and $E \defeq E(G)$. For a subset $U \subseteq V$, we use $\c{U}$ to denote the complement of $U$ in $V$, i.e., $\c{U} \defeq V \setminus U$. For $u \in V$ and $U \subseteq V$, set
	$$
	\deg(u) \defeq \deg_G(u) \;\;\;\;\;\text{and}\;\;\;\;\; \deg_U(u) \defeq |U \cap N_G(u)|.
	$$
	For $u \in V$, set
	$$
	\epsilon(u) \defeq \deg(u) - k,
	$$
	and for $U \subseteq V$, define
	$$
	\epsilon(U) \defeq \sum_{u \in U} \epsilon(u).
	$$
	Note that~\eqref{eq:main_inequality} is equivalent to 
	\begin{equation}\label{eq:epsilon_bound}
	\epsilon(V) \leq k-2.
	\end{equation}
	Since $G$ is $\Cov{H}$-critical, we have $\delta(G) \geq k$, i.e., $\epsilon(u) \geq 0$ for all~$u \in V$. Let
	$$
		D \defeq \{u \in V \,:\, \deg(u) = k\} = \set{u \in V \,:\, \epsilon(u) = 0}.
	$$
	Since $\epsilon(u) \geq 1$ for every $u \in \c{D}$,~\eqref{eq:epsilon_bound} yields $$|\c{D}| \leq k-2.$$
	Corollary~\ref{corl:D2} implies that $G[D]$ is a GDP-forest. Furthermore, since $n \geq k+1$, $D \neq \0$.
	
	From now on, we refer to the vertices of $H$ as \emph{colors} and to the independent sets in $H$ as \emph{colorings}. For $I$, $I' \in \I(H)$, we say that $I'$ \emph{extends} $I$ if $I'\supseteq I$. For $I \in \I(H)$, let
	$$
	\dom(I) \defeq \set{u \in V \,:\, I \cap L(u) \neq \0}.
	$$
	Since $G$ is $\Cov{H}$-critical, there is no coloring $I$ with $\dom(I) = V$; but for every proper subset $U \subset V$, there exists a coloring $I$ with $\dom(I) = U$.
	
	For $I \in \I(H)$ and $u  \in \c{(\dom(I))}$, let
	$$L_I(u) \defeq L(u) \setminus N_H(I).$$
	In other words, $L_I(u)$ is the set of all colors available for $u$ in a coloring extending $I$.
	
	For $u \in V$ and $U \subseteq V$, let
	$$
	\phi_U(u) \defeq \deg_U(u) - \epsilon(u).
	$$
	In particular, if $u \in D$, then $\phi_U(u) = \deg_U(u)$. Note that
	$$
	\phi_U(u) = \deg_U(u) - (\deg(u) - k) = k - (\deg(u) - \deg_U(u)) = k - \deg_{\c{U}}(u).
	$$
	Therefore, if $I$ is a coloring such that $\dom(I) = \c{U}$, then for all $u \in U$,
	\begin{equation}\label{eq:phi}
	|L_I(u)| \geq \phi_U(u).
	\end{equation}
	
	\subsection{A property of GDP-forests}
	
	The following simple general property of GDP-forests will be quite useful:
	
	\begin{prop}\label{prop:kGDP}
		Let $F$ be a nonempty GDP-forest of maximum degree at most $k$ not containing a clique of size $k+1$. Then
		\begin{equation}\label{eq:GDP_count}\sum_{u\in V(F)}(k - \deg_F(u)) \geq k,\end{equation} with equality only if $F \cong K_1$ or $F \cong K_k$.
	\end{prop}
	\begin{proof}
		It suffices to establish the proposition for the case when $F$ is connected, i.e., a GDP-tree. If $F$ is $2$-connected, i.e., a clique or a cycle, then the statement follows via a simple calculation. It remains to notice that adding a leaf block to a GDP-tree of maximum degree at most $k$ cannot decrease the quantity on the left-hand side of~\eqref{eq:GDP_count}.
	\end{proof}
	
	\begin{corl}\label{corl:kGDP}
		Let $U \subseteq D$ be the vertex set of a connected component of $G[D]$. Then $$|E_G(U, \c{D})| \geq k,$$ with equality only if $G[U] \cong K_k$.
	\end{corl}
	\begin{proof}
		We have
		$$
			|E_G(U, \c{D})| = \sum_{u \in U} \deg_{\c{D}}(u) = \sum_{u \in U} \deg_{\c{U}}(u)  = \sum_{u \in U} (k - \deg_U(u)).
		$$
		By Proposition~\ref{prop:kGDP} applied to $G[U]$, the latter quantity is at least $k$, with equality only if $G[U] \cong K_1$ or $G[U] \cong K_k$. It remains to notice that $G[U]\not \cong K_1$, since $\deg(u) = k$ for each $u \in U$, while $|\c{D}| \leq k - 2$.
	\end{proof}
	
	\subsection{Enhanced vertices}
	
	The following definition will play a crucial role in our argument:
	
	\begin{defn}\label{defn:enhanced}
		Let $I$ be a coloring and let $U \defeq \c{(\dom(I))}$. A vertex $u \in U \cap D$ is \emph{enhanced} by~$I$, or $I$ \emph{enhances} $u$, if $|L_I(u)| > \deg_U(u)$.
	\end{defn}
	
	\begin{remk}
		Note that, in the context of Definition~\ref{defn:enhanced}, we always have $|L_I(u)| \geq \deg_U(u)$.
	\end{remk}
	
	The importance of Definition~\ref{defn:enhanced} stems from the following lemma:
	
	\begin{lemma}\label{lemma:spoilt}
		Let $I$ be a coloring and let $U \defeq \c{(\dom(I))}$.
		\begin{enumerate}[wide,label=\normalfont{(\roman*)}]
			\item\label{item:spoilt:ext} Suppose that $I'$ is a coloring extending $I$. Let $u \in \c{(\dom(I'))} \cap D$. If $u$ is enhanced by~$I$, then it is also enhanced by $I'$.
			\item\label{item:spoilt:use} Let $U' \subseteq U \cap D$ be a subset such that the graph $G[U']$ is connected. Suppose that~$U'$ contains a vertex enhanced by~$I$. Then $I$ can be extended to a coloring $I'$ with $\dom(I') = \c{U} \cup U'$.
			\item\label{item:spoilt:problem} Suppose that $I$ enhances at least one vertex in each component of $G[U \cap D]$. Then $I$ cannot be extended to a coloring $I'$ with $\dom(I') \supseteq \c{D}$.
		\end{enumerate}
	\end{lemma}
	\begin{proof}
		Since \ref{item:spoilt:ext} is an immediate corollary of the definition and \ref{item:spoilt:use} follows from Theorem~\ref{theo:DP_Brooks}, it only remains to prove~\ref{item:spoilt:problem}. To that end, suppose, under the assumptions of~\ref{item:spoilt:problem}, that $I'$ is a coloring extending $I$ with $\dom(I') \supseteq \c{D}$. Reducing $I'$ if necessary, we may arrange that $\dom(I') = \c{U} \cup \c{D}$. Then, by~\ref{item:spoilt:ext}, $I'$ enhances at least one vertex in each component of $G[U \cap D]$. Applying~\ref{item:spoilt:use} to each connected component of $G[U \cap D]$, we can extend $I'$ to a coloring of the entire graph~$G$; a contradiction.
	\end{proof}
	
	The next lemma gives a convenient sufficient condition under which a given coloring can be extended so that the resulting coloring enhances a particular vertex:
	
	\begin{lemma}\label{lemma:spoiling}
		Let $I$ be a coloring and let $U \defeq \c{(\dom(I))}$. Let $u \in U \cap D$ and suppose that $A \subseteq U \cap N_G(u)$ is an independent set in $G$. Moreover, suppose that
		$$\min\set{\phi_U(v)\,:\, v \in A} > 0\;\;\;\;\;\text{and}\;\;\;\;\;\sum_{v \in A} \phi_U(v) > \deg_U(u).$$
		Then there is a coloring $I'$ with $\dom(I') = \c{U} \cup A$ that extends $I$ and enhances $u$.
	\end{lemma}
	\begin{proof}
		Since $A \in \I(G)$ and for all $v \in A$, $\phi_U(v) > 0$ (and hence, by~\eqref{eq:phi}, $|L_I(v)| > 0$), any coloring~$I'$ with $\dom(I') \subseteq \c{U} \cup A$ can be extended to a coloring with domain $\c{U} \cup A$. Therefore, it suffices to find a coloring that extends $I$ and enhances~$u$ and whose domain is contained in $\c{U} \cup A$. 
		
		If $u$ is enhanced by $I$ itself, then we are done, so assume that $|L_I(u)| = \deg_U(u)$. If for some $v \in A$, there is $x \in L_I(v)$ with no neighbor in $L_I(u)$, then $u$ is enhanced by $I \cup \set{x}$, and we are done again. Thus, we may assume that for every $v \in A$, the matching $E_H(L_I(v), L_I(u))$ saturates $L_I(v)$. For each $v \in A$ and $x \in L_I(v)$, let $f(x)$ denote the neighbor of $x$ in $L_I(u)$. Since $\sum_{v \in A} \phi_U(v) > \deg_U(u)$, and hence, by~\eqref{eq:phi}, $\sum_{v \in A} |L_I(v)| > |L_I(u)|$, there exist distinct vertices $v$, $w \in A$ and colors $x \in L_I(v)$, $y \in L_I(w)$ such that $f(x) = f(y)$. Then $u$ is enhanced by the coloring $I \cup \set{x, y}$, and the proof is complete.
	\end{proof}
	
	\begin{corl}\label{corl:common_neighbor}
		Suppose that $u$, $u_1$, $u_2 \in D$ are distinct vertices such that $u u_1$, $u u_2 \in E$, while $u_1 u_2 \not \in E$. Then the graph $G[D] - u_1 - u_2$ is disconnected.
	\end{corl}
	\begin{proof}
		Note that, since $u$, $u_1$, $u_2 \in D$, we have
		$$
			\phi_V(u_1) = \phi_V(u_2) = k \;\;\;\;\;\text{and}\;\;\;\;\; \deg(u) = k,
		$$
		so, by Lemma~\ref{lemma:spoiling}, there exist $x_1 \in L(u_1)$ and $x_2 \in L(u_2)$ such that $u$ is enhanced by the coloring $\set{x_1, x_2}$. Since for all $v \in \c{D}$,
		$$|L_{\set{x_1, x_2}}(v)| \geq |L(v)| - |\set{x_1, x_2}| = k-2 \geq |\c{D}|,$$
		we can extend $\set{x_1, x_2}$ to a coloring $I$ with $\dom(I) = \set{u_1, u_2}\cup\c{D}$. Due to Lemma~\ref{lemma:spoilt}\ref{item:spoilt:problem}, at least one connected component of the graph $G[D]-u_1-u_2$ contains no vertices enhanced by~$I$. Since, by Lemma~\ref{lemma:spoilt}\ref{item:spoilt:ext}, $I$ enhances $u$, $G[D] - u_1 - u_2$ is disconnected, as desired.
	\end{proof}
	
	We will often apply Lemma~\ref{lemma:spoiling} in the form of the following corollary:
	
	\begin{corl}\label{corl:spoiling}
		Suppose that $u \in D$ and let $v_1$, $v_2 \in \c{D} \cap N_G(u)$ be distinct vertices such that $v_1v_2 \not \in E$. Let $U \subseteq D$ be any set such that $u \in U$ and the graph $G[U]$ is connected. Then
		\begin{align*}
		\begin{array}{rcl}
		\text{either} & \min\set{\phi_U(v_1), \phi_U(v_2)} &\leq 0,\\
		\text{or} & \phi_U(v_1) + \phi_U(v_2) &\leq \deg_U(u) + 2.
		\end{array}
		\end{align*}
	\end{corl}
	\begin{proof}
		Notice that
		$$
			\phi_{U \cup \set{v_1, v_2}} (v_i) = \phi_U(v_i) \;\;\;\text{for each } i \in \set{1,2}, \;\;\;\;\;\text{and}\;\;\;\;\; \deg_{U \cup \set{v_1, v_2}}(u) = \deg_U(u) + 2.
		$$
		Therefore, is the claim fails, then we can first fix any coloring $I$ with $\dom(I) = \c{(U \cup \set{v_1, v_2})}$, and then apply Lemma~\ref{lemma:spoiling} to extend it to a coloring $I'$ with $\dom(I') = \c{U}$ that enhances $u$. Since $G[U]$ is connected, such a coloring cannot exist by Lemma~\ref{lemma:spoilt}\ref{item:spoilt:problem}.
	\end{proof}	
	
	The following observation can be viewed as an analog of Lemma~\ref{lemma:spoilt}\ref{item:spoilt:use} for edges instead of vertices:
	
	\begin{lemma}\label{lemma:edge_extends}
		Let $I$ be a coloring and let $U \defeq \c{(\dom(I))}$. Let $U' \subseteq U \cap D$ be a subset such that the graph $G[U']$ is connected and let $u_1$, $u_2 \in U'$ be adjacent non-cut vertices in $G[U']$. Suppose that the matching $E_H(L_I(u_1), L_I(u_2))$ is not perfect. Then $I$ can be extended to a coloring $I'$ with $\dom(I') = \c{U} \cup U'$.
	\end{lemma}
	\begin{proof}
		Follows from Theorem~\ref{theo:DP_Brooks}.
	\end{proof}
	
	\subsection{Vertices of small degree}
	
	In this subsection we establish some structural properties that $G$ must possess if the minimum degree of the graph $G[D]$ is ``small'' (namely at most~$2$).
	
	\begin{lemma}\label{lemma:no_small_degrees}
		\begin{enumerate}[wide,label=\normalfont{(\roman*)}]
			\item \label{item:no_small_degrees:2} The minimum degree of $G[D]$ is at least $2$.
			\item \label{item:no_small_degrees:bounds} If there is a vertex $u \in D$ such that $\deg_D(u) = 2$, then $|\c{D}| = k-2$, $u$ is adjacent to every vertex in~$\c{D}$, and $\epsilon(v) = 1$ for all $v \in \c{D}$.
			\item \label{item:no_small_degrees:structure} If the graph $G[D]$ has a connected component with at least $3$ vertices of degree $2$, then $G[\c{D}]$ is a disjoint union of cliques.
			\item \label{item:no_small_degrees:more_structure} If the graph $G[D]$ has a connected component with at least $4$ vertices of degree $2$, then $G[\c{D}] \cong K_{k-2}$.
		\end{enumerate}
	\end{lemma}
	\begin{proof}
		\ref{item:no_small_degrees:2} For each $u \in D$, we have
		$$
			k-2 \geq |\c{D}| \geq \deg_{\c{D}}(u) =  k - \deg_D(u),
		$$
		so $\deg_D(u) \geq 2$.
		
		\ref{item:no_small_degrees:bounds} If $u \in D$ and $\deg_D(u) = 2$, then $u$ has exactly $k-2$ neighbors in $\c{D}$. Thus,
		$$\epsilon(\c{D}) = |\c{D}| = k-2,$$ which implies all the statements in~\ref{item:no_small_degrees:bounds}.
		
		\ref{item:no_small_degrees:structure} Let $U \subseteq D$ be the vertex set of a connected component of $G[D]$ such that $G[U]$ contains at least $3$ vertices of degree $2$. Suppose, towards a contradiction, that $G[\c{D}]$ is not a disjoint union of cliques, i.e., there exist distinct vertices $v_0$, $v_1$, $v_2 \in \c{D}$ such that $v_0v_1$, $v_0v_2 \in E$, while $v_1v_2 \not \in E$. By~\ref{item:no_small_degrees:bounds}, each vertex in $\c{D}$ is adjacent to every vertex of degree $2$ in $G[D]$, $|\c{D}| = k-2$, and $\epsilon(v) = 1$ for all $v \in \c{D}$. Thus,
		$$
		\phi_{U \cup \set{v_0, v_1, v_2}} (v_i) = \deg_{U\cup \set{v_0, v_1, v_2}}(v_i) - \epsilon(v_i) \geq (3 + 1) -1 =  3 \;\;\; \text{for each } i \in \set{1,2}.
		$$
		Fix any vertex $u \in U$ such that $\deg_U(u) = 2$. Then
		$$
		\deg_{U \cup \set{v_0, v_1, v_2}} (u) = 2 + 3 = 5.
		$$
		Therefore, by Lemma~\ref{lemma:spoiling}, there exists a coloring $I$ with domain
		$$
			\dom(I) = \c{(U \cup \set{v_0, v_1, v_2})} \cup \set{v_1, v_2} = \c{(U \cup \set{v_0})}
		$$
		that enhances~$u$. By~\eqref{eq:phi},
		$$|L_I(v_0)| \geq \phi_U(v_0) = \deg_U(v_0) - \epsilon(v_0) \geq 3 - 1 = 2 > 0,$$
		so $I$ can be extended to a coloring $I'$ with $\dom(I') = \c{U}$. This contradicts Lemma~\ref{lemma:spoilt}\ref{item:spoilt:problem}.
		
		\ref{item:no_small_degrees:more_structure} If $U \subseteq D$ is the vertex set of a connected component of $G[D]$ with at least $4$ vertices of degree $2$ and $v_1$, $v_2 \in \c{D}$ are distinct nonadjacent vertices, then we have
		$$
		\phi_{U}(v_i)  = \deg_U(v_i) - \epsilon(v_i) \geq 4 - 1 = 3 \;\;\; \text{for each }i \in \set{1,2},
		$$
		so for every vertex $u \in U$ with $\deg_U(u) = 2$, we have
		$$
		\phi_U(v_1) + \phi_U(v_2) \geq 3 + 3 > 4 = \deg_U(u)  +2;
		$$
		a contradiction to Corollary~\ref{corl:spoiling}.
	\end{proof}
	
	\subsection{Terminal sets}
	
	We start this section by introducing some definitions and  notation that will be used throughout the rest of the proof.
	
	\begin{defn}
		A \emph{terminal set} is a subset $B \subseteq D$ such that $G[B]$ is a leaf block in a connected component of $G[D]$. For a terminal set $B$, $C_B \supseteq B$ denotes the vertex set of the connected component of $G[D]$ that contains $B$. A vertex $u \in D$ is \emph{terminal} if it belongs to some terminal set $B$ and is not a cut-vertex in $G[C_B]$.
	\end{defn}
	
	By definition, a terminal set contains at most one non-terminal vertex. Since $G[D]$ is a GDP-forest, if $B$ is a terminal set, then $G[B]$ is either a cycle or a clique. By Lemma~\ref{lemma:no_small_degrees}\ref{item:no_small_degrees:2}, the cardinality of a terminal set is at least $3$.
	
	\begin{defn}
		A terminal set $B$ is \emph{dense} if $G[B]$ is not a cycle; otherwise,  $B$ is \emph{sparse}.
	\end{defn}
	
	By definition, the cardinality of a dense terminal set is at least $4$.
	
	Our proof hinges on the following key fact:
	
	\begin{lemma}\label{lemma:not_all_sparse}
		There exists a dense terminal set.
	\end{lemma}
	
	Before proving Lemma~\ref{lemma:not_all_sparse}, we need the following simple observation:
	
	\begin{prop}\label{prop:wheels}
		Let $W_4$ denote the $4$-wheel. 
		Then $\chi_{DP}(W_4) = 3$.
	\end{prop}
	\begin{proof}
		Let $\Cov{F} = (M, F)$ be a $3$-fold cover of $W_4$ and suppose that~$W_4$ is not $\Cov{F}$-colorable. Let $v \in V(W_4)$ denote the center of $W_4$ and let $U \defeq V(W_4)\setminus \set{v}$ (so $W_4[U]$ is a $4$-cycle). Define a function $f \colon V(F) \to M(v)$ by
		\[
		f(x) = y \,\vcentcolon\Longleftrightarrow\, (x = y) \text{ or } (x\not \in M(v) \text{ and }xy \in E(F)).
		\]
		Since $\deg_{W_4}(u) = 3$ for all $u \in U$, Theorem~\ref{theo:DP_Brooks} implies that $f$ is well-defined. Since $W_4$ is $3$-colorable (in the sense of ordinary graph coloring), there exist an edge $u_1u_2 \in E(W_4)$ and a pair of colors $x_1 \in M(u_1)$, $x_2 \in M(u_2)$ such that $x_1 x_2 \in E(F)$ and $f(x_1) \neq f(x_2)$. Note that $u_1 \neq v$ since otherwise $f(x_1) = x_1 = f(x_2)$ by definition. Similarly, $u_2 \neq v$, so $\set{u_1, u_2} \subset U$. Let $y \defeq f(x_2)$. Then $x_1$ has no neighbor in $M(u_2) \setminus N_{F}(y)$, and hence $\set{y}$ can be extended to an $\Cov{F}$-coloring of $W_4$; a contradiction.
	\end{proof}
	
	\begin{proof}[Proof of Lemma~\ref{lemma:not_all_sparse}]
		Suppose that every terminal set is sparse. Since every terminal set induces a cycle, each component of $G[D]$ contains at least $3$ vertices of degree~$2$, and a component of $G[D]$ with exactly~$3$ vertices of degree $2$ must be isomorphic to a triangle. By Lemma~\ref{lemma:no_small_degrees}\ref{item:no_small_degrees:bounds}, each vertex in $\c{D}$ is adjacent to every vertex of degree $2$ in $G[D]$, $|\c{D}| = k-2$, and $\epsilon(v) = 1$ for all $v \in \c{D}$. Furthermore, by Lemma~\ref{lemma:no_small_degrees}\ref{item:no_small_degrees:structure}\ref{item:no_small_degrees:more_structure}, $G[\c{D}]$ is a disjoint union of cliques and, unless every  component of $G[D]$ is isomorphic to a triangle, $G[\c{D}] \cong K_{k-2}$.
		
		\begin{claim}\label{claim:not_clique}
			$G[\c{D}] \not \cong K_{k-2}$.
		\end{claim}
		\begin{claimproof}
			Assume, towards a contradiction, that $G[\c{D}] \cong K_{k-2}$. Then every vertex in $\c{D}$ has exactly $(k+1) - (k-3) = 4$ neighbors in $D$. Therefore, the number of vertices of degree~$2$ in $G[D]$ is at most $4$. Since every  component of $G[D]$ contains at least $3$ vertices of degree $2$, the graph $G[D]$ is connected. Since $|D| \geq 4$, $G[D]$ is not a triangle. Thus, it contains precisely $4$ terminal vertices of degree $2$; i.e., it either is a $4$-cycle, or contains exactly two leaf blocks, both of which are triangles.
			
			{\sc Case~$1$}: \emph{$G[D]$ is a $4$-cycle}. We will show that in this case $G$ is $\Cov{H}$-colorable. Choose any vertex $v \in \c{D}$ and let $W \defeq G[\set{v}\cup D]$. Note that $W$ is a $4$-wheel. Fix an arbitrary coloring $I \in \I(H)$ with $\dom(I) = \c{(\set{v} \cup D)}$. For all $u \in \set{v} \cup D$, we have $|L_I(u)| \geq k - (k-3) = 3$, so by Proposition~\ref{prop:wheels}, $I$ can be extended to an $\Cov{H}$-coloring of the entire graph $G$.
			
			{\sc Case~$2$:} \emph{$G[D]$ contains exactly two leaf blocks, both of which are triangles}. Since each vertex in~$\c{D}$ has only $4$ neighbors in $D$, every non-terminal vertex in $D$ has degree $k$ in $G[D]$. Notice that every vertex of degree $k$ in $G[D]$ is a cut-vertex. Indeed, if a vertex $u \in D$ is not a cut-vertex in $G[D]$, then the degree of any cut-vertex in the same block as $u$ strictly exceeds the degree of $u$ (since the blocks of the   GDP-tree $G[D]$  are regular graphs). Thus, either the two terminal triangles share a cut-vertex (and, in particular, $k=4$), or else, their cut-vertices are joined by an edge (and $k=3$). The former option  contradicts Corollary~\ref{corl:common_neighbor}; the latter one implies~$G \in \D_3$.
		\end{claimproof}
		
		By Claim~\ref{claim:not_clique}, $G[\c{D}]$ is a disjoint union of at least $2$ cliques. In particular, every connected component of $G[D]$ is isomorphic to a triangle. Suppose that $G[D]$ has $\ell$ connected components (so $|D| = 3\ell$). If a vertex $v \in \c{D}$ belongs to a component of $G[\c{D}]$ of size $r$, then its degree in $G$ is precisely $(r-1) + 3\ell$. On the other hand, $\deg(v) = k+1$. Thus, $k+1 = (r-1) + 3\ell$, i.e., $r = k - 3\ell + 2$. In particular, $|\c{D}| = k-2$ is divisible by $k-3\ell + 2$, so $\ell \geq 2$.
		
		{\sc Case~1:} \emph{The set $\c{D}$ is not independent, i.e., $k-3\ell + 2 \geq 2$}. Let $T_1$, $T_2 \subset D$ (resp. $C_1$, $C_2 \subset \c{D}$) be the vertex sets of any two distinct connected components of $G[D]$ (resp. $G[\c{D}]$). For each $i \in \set{1,2}$, fix a vertex $u_i \in T_i$ and a pair of distinct vertices $v_{i1}$, $v_{i2} \in C_i$. Set $U \defeq T_1 \cup T_2 \cup \set{v_{11}, v_{12}, v_{21}, v_{22}}$ and let $I \in \I(H)$ be such that $\dom(I) = \c{U}$. Note that
		$$
		\phi_U(v_{11}) = \phi_U(v_{21}) = 7-1 = 6,
		$$
		while
		$$
		\deg_U(u_1) = 6,
		$$
		so, by Lemma~\ref{lemma:spoiling}, there exist $x_{11} \in L_I(v_{11})$ and $x_{21} \in L_I(v_{21})$ such that $$I' \defeq I \cup \set{x_{11}, x_{21}}$$ is a coloring that enhances $u_1$. Now, upon setting $U' \defeq U \setminus \set{v_{11}, v_{21}}$, we obtain
		$$
		\phi_{U'}(v_{12}) = \phi_{U'}(v_{22}) = 6-1 = 5,
		$$
		while
		$$
		\deg_{U'} (u_2) = 4,
		$$
		so, by Lemma~\ref{lemma:spoiling} again, we can choose $x_{12} \in L_{I'}(v_{12})$ and $x_{22} \in L_{I'}(v_{22})$ so that $$I'' \defeq I' \cup \set{x_{12}, x_{22}}$$ is a coloring that enhances both $u_1$ and $u_2$. However, the existence of such $I''$ contradicts Lemma~\ref{lemma:spoilt}\ref{item:spoilt:problem}.
		
		{\sc Case~2:} \emph{The set $\c{D}$ is independent, i.e., $k - 3\ell + 2 = 1$}. In other words, we have $k = 3\ell -1$. Since $\ell \geq 2$, we get $k \geq 6-1 = 5$, so $|\c{D}| = k-2 \geq 3$. Let $v_1$, $v_2$, $v_3 \in \c{D}$ be any three distinct vertices in $\c{D}$ and let $T \subset D$ be the vertex set of any connected component of $G[D]$. Fix a vertex $u \in T$, set $U \defeq T \cup \set{v_1, v_2, v_3}$, and let $I \in \I(H)$ be such that $\dom(I) = \c{U}$. Note that
		$$
		\phi_U(v_1) = \phi_U(v_2) = \phi_U(v_3) = 3-1 = 2,
		$$
		while
		$$
		\deg_U(u) = 5.
		$$
		Therefore, by Lemma~\ref{lemma:spoiling}, we can choose $x_1 \in L_I(v_1)$, $x_2 \in L_I(v_2)$, and $x_3\in L_I(v_3)$ so that $$I' \defeq I \cup \set{x_1, x_2, x_3}$$ enhances $u$. This observation contradicts Lemma~\ref{lemma:spoilt}\ref{item:spoilt:problem} and finishes the proof.
	\end{proof}
	
	\section{Dense terminal sets and their neighborhoods}
	
	\subsection{Outline of the proof}
	
	Lemma~\ref{lemma:not_all_sparse} asserts that at least one terminal set is dense. In this section we explore the structural consequences of this assertion and eventually arrive at a contradiction.
	
	\begin{defn}
		Let $B$ be a terminal set. Let $S_B$ denote the set of all vertices in $\c{B}$ that are adjacent to every vertex in $B$ and let $T_B \defeq N_G(B) \setminus (B \cup S_B)$.
	\end{defn}
	
	By definition, $S_B\subseteq \c{D}$; however, if $B \neq C_B$, then $T_B \cap D \neq \0$.
	
	The following statement will be used several times throughout the rest of the argument:
	
		\begin{lemma}\label{lemma:many_neighbors_outside}
			Let $B$ be a dense terminal set and let $v \in T_B$. Then $v$ has at least $k-1$ neighbors outside of $B$. If, moreover, there exist terminal vertices $u_0$, $u_1 \in B$ such that $u_0v \not\in E$, $u_1v \in E$, then $v$ has at least $k-1$ neighbors outside of $C_B$.
		\end{lemma}
		\begin{proof}
			Let $u_0$, $u_1 \in B$ be such that $u_0v \not \in E$ and $u_1v \in E$. If one of $u_0$, $u_1$ is not terminal, then set $U \defeq B$; otherwise, set $U \defeq C_B$. Our goal is to show that $v$ has at least $k-1$ neighbors outside of $U$. Assume, towards a contradiction, that $\deg_{\c{U}}(v) \leq k-2$. Let $I \in \I(H)$ be such that $\dom(I) = \c{(U \cup \set{v})}$. By~\eqref{eq:phi}, we have $$|L_I(v)| \geq \phi_U(v) \geq k - (k-2) = 2,$$ so let $x_1$, $x_2$ be any two distinct elements of $L_I(v)$. Since $u_0 v \not\in E$, we have $$L_{I \cup \set{x_1}}(u_0) = L_{I\cup\set{x_2}}(u_0) = L_I(u_0),$$ so, by Lemma~\ref{lemma:edge_extends}, the matching $E_H(L_I(u_0), L_{I\cup\set{x_i}}(u_1))$ is perfect for each $i \in \set{1,2}$. This implies that the unique vertex in $L_I(u_1)$ that has no neighbor in $L_I(u_0)$ is adjacent to both $x_1$ and $x_2$, which is impossible.
		\end{proof}
	
	The rest of the proof of Theorem~\ref{theo:sharp_DP_Dirac} proceeds as follows. 
	Consider a dense terminal set~$B$. Roughly speaking, Lemma~\ref{lemma:many_neighbors_outside} asserts that the vertices in $T_B$ must have ``many'' neighbors outside of $B$. Since the degrees of the vertices in $\c{D}$ cannot be too big, the vertices in $T_B$ should only have ``very few'' neighbors in $B$. This implies that ``most'' edges between $B$ and $\c{D}$ actually connect $B$ with $S_B$. This intuition guides the proof of Corollary~\ref{corl:S}, which asserts that $G[B \cup S_B]$ is a clique of size $k$ (however, the proof of Lemma~\ref{lemma:S_is_a_clique}, the main step towards Corollary~\ref{corl:S}, is somewhat lengthy and technical).
	
	The fact that $G$ is a \emph{minimum} counterexample to Theorem~\ref{theo:sharp_DP_Dirac} is only used once during the course of the proof, namely in establishing Lemma~\ref{lemma:T}, which claims that for a dense terminal set~$B$, the graph $G[T_B]$ is a clique. The proof of Lemma~\ref{lemma:T} is also the only time when it is important to work in the more general setting of DP-colorings rather than just with list colorings. The proof proceeds by assuming, towards a contradiction, that there exist two nonadjacent vertices $v_1$, $v_2 \in T_B$, and letting $G^\ast$ be the graph obtained from $G$ by removing $B$ and adding an edge between $v_1$ and $v_2$. Since $G^\ast$ has fewer vertices than $G$, it cannot contain a counterexample to Theorem~\ref{theo:sharp_DP_Dirac} as a subgraph. This fact can be used to eventually arrive at a contradiction. En route to that goal we study the properties of a certain cover $\Cov{H}^\ast$ of~$G^\ast$, and that cover is not necessarily induced by a
	list assignment, even if $\Cov{H}$ is.
	
	With Lemma~\ref{lemma:T} at hand, we can pin down the structure of $G[S_B \cup T_B]$ very precisely, which is done in Lemmas~\ref{lemma:|T|_is_2} and~\ref{lemma:entire} and in Corollary~\ref{corl:D^c}. The restrictiveness of these results precludes having ``too many'' dense terminal sets; this is made precise by Lemma~\ref{lemma:exists_sparse}, which asserts that at least one terminal set is sparse. However, due to Lemma~\ref{lemma:no_small_degrees}, having a sparse terminal set leads to its own restrictions on the structure of $G[\c{D}]$, which finally yield a contradiction that finishes the proof of Theorem~\ref{theo:sharp_DP_Dirac}.
	
	\subsection{The set $S_B$ is large}\label{subsec:S_is_large}
	
	In this section we prove that for any dense terminal set $B$, $|S_B| \geq k - |B|$ (see Lemma~\ref{lemma:S}).
	
	\begin{lemma}\label{lemma:S_prime}
		Let $B$ be a dense terminal set. If $|S_B| \leq k - |B| - 1$, then the following statements hold:
		\begin{enumerate}[wide,label=\normalfont{(\roman*)}]
			\item\label{item:S:first}\label{item:S_prime:S} $|S_B| = k - |B| - 1$;
			\item \label{item:S_prime:union} $\c{D} = S_B \cup (T_B \cap \c{D})$;
			\item \label{item:S_prime:k-2} $|\c{D}| = |S_B| + |T_B \cap \c{D}| = k - 2$, and thus $\epsilon(v) = 1$ for every $v \in \c{D}$;
			\item \label{item:S_prime:T} every vertex in $T_B \cap \c{D}$ has exactly $k-1$ neighbors outside of $B$; and
			\item\label{item:S:last} \label{item:S_prime:cut} $B \neq C_B$, and the cut vertex $u_0 \in B$ of $G[C_B]$ has no neighbors in $T_B \cap \c{D}$.
		\end{enumerate}
	\end{lemma}
	\begin{proof}
		Let $S \defeq S_B$ and let $T \defeq T_B\cap \c{D}$.
		Set $b \defeq |B|$, $s \defeq |S|$, and $t \defeq |T|$. Suppose that $s \leq k-b-1$. Since each terminal vertex in $B$ has exactly $k-(b-1)-s$ neighbors in $T$, the number of edges between $B$ and $T$ is at least $(b-1)(k-(b-1)-s)$. Also, by 
		Lemma~\ref{lemma:many_neighbors_outside}, each vertex in $T$ has at least $k-1$ neighbors in $\c{B}$. Hence,
		\begin{align*}
		\epsilon(\c{D}) & \geq \epsilon(S) + \epsilon(T) \\
		&\geq s + (b-1)(k-(b-1)-s) + (k-1)t - kt \\
		&= s + (b-1)(k-(b-1)-s) - t.
		\end{align*}
		Note that $s + t \leq |\c{D}| \leq k-2$, so $t \leq k-2-s$. Therefore,
		\begin{align*}
		s + (b-1)(k-(b-1)-s) - t \geq 2s + (b-1)(k-(b-1)-s) -k+2.
		\end{align*}
		Since $b \geq 4$, the last expression is decreasing in $s$, and hence
		\begin{align*}
		&2s + (b-1)(k-(b-1)-s) -k+2 \\
		\geq\, &2(k-b-1) + (b-1)(k - (b-1) - (k-b-1)) -k+2 \\
		=\, &k-2.	
		\end{align*}
		On the other hand, $\epsilon(\c{D}) \leq k-2$. This implies that none of the above inequalities can be strict, yielding~\ref{item:S:first}--\ref{item:S:last}.
	\end{proof}	
	
	\begin{lemma}\label{lemma:S1}
		Let $B$ be a dense terminal set. Suppose that $v_1$, $v_2 \in S_B$ are distinct vertices such that $v_1 v_2 \not \in E$. Then the following statements hold:
		\begin{enumerate}[wide,label=\normalfont{(\roman*)}]
			\item\label{item:S1:first}\label{item:S1:D^c} $|\c{D}| = k - |B| + 1$;
			\item \label{item:S1:v12} $\epsilon(v_1) + \epsilon(v_2) = |B|-1$; and
			\item\label{item:S1:last}\label{item:S1:epsilon} $\epsilon(v) = 1$ for every $v \in \c{D} \setminus \set{v_1, v_2}$.
		\end{enumerate}
	\end{lemma}
	\begin{proof}
		Let $b \defeq |B|$. Each terminal vertex $u \in B$ has exactly $k-b+1$ neighbors in~$\c{D}$; in particular, $|\c{D}| \geq k-b+1$. By Corollary~\ref{corl:spoiling}, we have
		\begin{align*}
		\begin{array}{rcl}
		\text{either} & \min\set{\phi_B(v_1), \phi_B(v_2)} &\leq 0,\\
		\text{or} & \phi_B(v_1) + \phi_B(v_2) &\leq b+1.
		\end{array}
		\end{align*}
		In the case when $\phi_B(v_i) \leq 0$ for some $i \in \set{1,2}$, we have $\epsilon(v_i) = b - \phi_B(v_i) \geq b$, so
		\begin{equation}\label{ej32}
		\epsilon(v_1) + \epsilon(v_2) \geq b.
		\end{equation}
		In the other case, i.e., when $\phi_B(v_1) + \phi_B(v_2) \leq b+1$, we get
		$$
		\epsilon(v_1) + \epsilon(v_2) = (b - \phi_B(v_1)) + (b - \phi_B(v_2)) \geq b-1.
		$$
		Hence
		\begin{equation}\label{tr1}
		\epsilon(\c{D})\geq \epsilon(v_1) + \epsilon(v_2)+|\c{D}\setminus\{v_1,v_2\}| \geq (b-1) + (k-b-1) = k-2.
		\end{equation}
		Since $\epsilon(\c{D}) \leq k-2$,~\eqref{ej32} does not hold and none of the  inequalities in~\eqref{tr1} can be strict, yielding~\ref{item:S1:first}--\ref{item:S1:last}.
	\end{proof}
	
	\begin{lemma}\label{lemma:S}
		Let $B$ be a dense terminal set. Then $|S_B| \geq k - |B|$.
	\end{lemma}
	\begin{proof}
		Let $S \defeq S_B$ and let $T \defeq T_B\cap \c{D}$.
		Set $b \defeq |B|$, $s \defeq |S|$, and $t \defeq |T|$. Suppose that $s \leq k-b-1$. Then, by Lemma~\ref{lemma:S_prime}\ref{item:S_prime:S}, $s = k - b - 1$. We claim that $G[S]$ is a clique. Indeed, otherwise, by Lemma~\ref{lemma:S1}\ref{item:S1:D^c}, $|\c{D}| = k - b + 1$; on the other hand, by Lemma~\ref{lemma:S_prime}\ref{item:S_prime:k-2}, $|\c{D}| = k-2$, so we get $k-2 = k - b+1$, i.e., $b = 3$, which contradicts the fact that $B$ is dense.
		
		By Lemma~\ref{lemma:S_prime}\ref{item:S_prime:k-2}, the degree of every vertex in $\c{D}$ is exactly $k+1$. Since each vertex in $S$ has $b$ neighbors in $B$ and $s - 1 = k - b - 2$ neighbors in $S$, it has exactly $(k+1) - b - (k - b - 2) = 3$ neighbors in $\c{(B \cup S)}$.
		
		By Lemma~\ref{lemma:S_prime}\ref{item:S_prime:cut}, $B \neq C_B$. Let $u_0$ denote the cut vertex in $B$ and let $B'$ be any terminal subset of $C_B$ distinct from $B$. Set $b' \defeq |B'|$.
		
		By Lemma~\ref{lemma:S_prime}\ref{item:S_prime:union}, $t = |\c{D}| - s = (k - 2) - (k - b - 1) = b - 1 \geq 3$; in particular, $T \neq \0$. Due to Lemma~\ref{lemma:S_prime}\ref{item:S_prime:T}\ref{item:S_prime:cut}, every vertex in $T$ has exactly $k-1$ neighbors in $\c{B}$ and is not adjacent to $u_0$. Together with Lemma~\ref{lemma:S_prime}\ref{item:S_prime:k-2}, this implies that each vertex in $T$ has exactly $(k+1) - (k-1) = 2$ neighbors in $B \setminus \set{u_0}$. We have $|B \setminus \set{u_0}| = b-1 \geq 3$, so, by Lemma~\ref{lemma:many_neighbors_outside}, every vertex in $T$ has $k-1$ neighbors outside of $C_B$. Therefore, there are no edges between $T$ and $C_B \setminus B$; in particular, there are no edges connecting $T$ to the terminal vertices in $B'$.
		
		Consider any terminal vertex $u \in B'$. Since $T \neq \0$ and no edges connect $u$ and $T$, $\deg_{B'}(u) > 2$; therefore, $B'$ is a dense terminal set. By Lemma~\ref{lemma:S_prime}\ref{item:S_prime:union}, $\c{D} = S \cup T$, so $u$ has exactly $k - b' + 1$ neighbors in $S$. Thus, $k - b - 1 = s \geq k - b' + 1$, i.e., $b' \geq b + 2 \geq 6$. Let $v$ be any neighbor of $u$ in $S$. Since $v$ has only $3$ neighbors in $\c{(B \cup S)}$ and $b' > 4$, there exists another terminal vertex $u' \in B'$ such that $u'v \not \in E$. By Lemma~\ref{lemma:many_neighbors_outside}, $v$ has at least $k-1$ neighbors outside of $C_{B'} = C_B$. Of those, $s - 1$ belong to $S$; since $v$ has only $3$ neighbors outside of $B \cup S$ and is adjacent to $u$, it has at most $3 - 1 = 2$ neighbors in $\c{(C_B \cup S)}$. Hence, $k-1 \leq (s-1) + 2 = (k - b - 2) + 2 = k - b$, i.e., $b \leq 1$, which is impossible.
	\end{proof}
	
	\subsection{The graph $G[S_B]$}\label{subsec:S_is_a_clique}
	
	\begin{lemma}\label{lemma:S_is_a_clique}
		Let $B$ be a dense terminal set. Then $G[S_B]$ is a clique.
	\end{lemma}
	\begin{proof}
		Let $S \defeq S_B$ and suppose that $G[S]$ is not a clique, i.e., there exist distinct $v_1$, $v_2\in S$ such that $v_1 v_2 \not \in E$. Without loss of generality, we may assume that $\deg(v_1) \geq \deg(v_2)$. 
		We will proceed via a series of claims, establishing a precise structure of $G[\c{D}]$, which will eventually lead to a contradiction. For the rest of the proof, we set $b \defeq |B|$ and $s \defeq |S|$.
		
		Recall that, by Lemma~\ref{lemma:S1}, we have the following:
		\begin{enumerate}[wide,label=\normalfont{(\roman*)}]
			\item $|\c{D}| = k - b + 1$;
			\item $\epsilon(v_1) + \epsilon(v_2) = b-1$; and
			\item $\epsilon(v) = 1$ for every $v \in \c{D} \setminus \set{v_1, v_2}$.
		\end{enumerate}
		
		\begin{claim}
			$\c{D} = S$ and $B = C_B$.
		\end{claim}
		\begin{claimproof}
			Suppose, towards a contradiction, that there is a vertex $v \in \c{D} \setminus S$. Since, by Lemma~\ref{lemma:S1}\ref{item:S1:D^c}, $|\c{D}| = k-b+1$, each terminal vertex in $B$ is adjacent to every vertex in~$\c{D}$. Therefore, $\deg_B(v) = b-1$ and, due to Lemma~\ref{lemma:many_neighbors_outside}, $\deg_{\c{B}}(v) \geq k-1$. Then
			$$\epsilon(v) = \deg(v) - k \geq (b-1) + (k-1) - k = b-2 > 1;$$
			a contradiction to Lemma~\ref{lemma:S1}\ref{item:S1:epsilon}.
			
			Since $|S| = |\c{D}| = k-b+1$, every vertex in $B$ has $(b-1) + (k-b+1) = k$ neighbors in $B \cup S$, so there are no edges between $B$ and $D \setminus B$; therefore, $B = C_B$.
		\end{claimproof}
		
		\begin{claim}\label{claim:deg2}
			The graph $G[D]$ has no vertices of degree $2$.
		\end{claim}
		\begin{claimproof}
			Indeed, otherwise Lemma~\ref{lemma:no_small_degrees} would yield $|\c{D}| = k-2$. Since $|\c{D}| = k-b+1$, this implies $b = 3$, contradicting the denseness of $B$.
		\end{claimproof}
		
		\begin{claim}\label{claim:at_least_3}
			$s \geq 3$, i.e., $b \leq k-2$.
		\end{claim}
		\begin{claimproof}
			Suppose, towards a contradiction, that $s = 2$, i.e., $S = \set{v_1, v_2}$. We will argue that in this case $G \in \D_k$. Since, by Lemma~\ref{lemma:S1}\ref{item:S1:D^c}, $s = k - b + 1$, we have $b = k-1$. In particular, since $b \geq 4$, we have $k \geq 5$. By Lemma~\ref{lemma:S1}\ref{item:S1:v12}, $\epsilon(S) = b-1 = k-2$, so there are exactly $(k-2) + 2k - 2(k-1) = k$ edges between $S$ and $D \setminus B$. Let $U$ be any connected component of $G[D]$ distinct from $B$. By Corollary~\ref{corl:kGDP}, the number of edges between $U$ and $S$ is at least~$k$, with equality only if $G[U] \cong K_k$; therefore, $D\setminus B = U$, and we indeed have $G[U] \cong K_k$. Then every vertex in $U$ has exactly one neighbor in $S$ and each vertex in $S$ has at least two neighbors in $U$ (for its degree is at least $k+1$), yielding $G \in \D_k$, as desired.
		\end{claimproof}
		
		\begin{claim}\label{claim:S-v1}
			$G[S \setminus \set{v_1}]$ is a clique.
		\end{claim}
		\begin{claimproof}
			Suppose that for some distinct $w_1$, $w_2 \in S \setminus \set{v_1}$, we have $w_1 w_2 \not \in E$. Applying Lemma~\ref{lemma:S1}\ref{item:S1:epsilon} with $w_1$ and $w_2$ in place of $v_1$ and $v_2$, we obtain $\epsilon(v_1) = 1$. Since, by our choice, $\deg(v_1) \geq \deg(v_2)$, and thus $\epsilon(v_1) \geq \epsilon(v_2)$, we get $\epsilon(v_2) = 1$ as well. But then $2 = \epsilon(v_1) + \epsilon(v_2) = b-1$, i.e., $b = 3$; a contradiction.
		\end{claimproof}
		
		\begin{claim}\label{claim:not_both}
			$\deg_S(v_1)=0$.
		\end{claim}
		\begin{claimproof}
			Suppose  that $v \in S \setminus \set{v_1, v_2}$ is adjacent to $v_1$. Note that by Claim~\ref{claim:S-v1}, $v$ is also adjacent to $v_2$. Let $U \defeq B \cup \set{v_1, v_2, v}$ and let $u$ be any vertex in $B$. Note that
			$$
			\deg_U(u) = (b-1) + 3 = b+2.
			$$
			On the other hand, since $\epsilon(v_1) + \epsilon(v_2) = b-1$, for each $i \in \set{1,2}$, we have $\epsilon(v_i) \leq b-2$, so
			$$
			\phi_U(v_i) = (b+1) - \epsilon(v_i) \geq (b+1) - (b-2) = 3 > 0;
			$$
			moreover,
			$$
			\phi_U(v_1) + \phi_U(v_2) = 2(b+1) - (b-1) = b+3 > b+2.
			$$
			Therefore, by Lemma~\ref{lemma:spoiling}, for any $I \in \I(H)$ with $\dom(I) = \c{U}$, we can find $x_1 \in L_I(v_1)$ and $x_2 \in L_I(v_2)$ such that $u$ is enhanced by $I' \defeq I \cup \set{x_1, x_2}$. Note that
			$$|L_{I'}(v)| \geq \phi_B(v) = b - 1 > 0,$$ so $I'$ can be extended to a coloring with domain $\c{B}$, which contradicts Lemma~\ref{lemma:spoilt}\ref{item:spoilt:problem}.
		\end{claimproof}
		
		\begin{claim}\label{claim:epsilon}
			$\epsilon(v_1) = b - 2$ and $\epsilon(v) = 1$ for all $v \in S \setminus \set{v_1}$.
		\end{claim}
		\begin{claimproof}
			Consider any $v \in S \setminus \set{v_1}$. By Claim~\ref{claim:at_least_3}, we can choose some $v' \in S \setminus \set{v_1, v}$. Due to Claim~\ref{claim:not_both}, $v_1v' \not \in E$, so we can apply Lemma~\ref{lemma:S1}\ref{item:S1:epsilon} with $v'$ in place of $v_2$ to obtain $\epsilon(v) = 1$. In particular, $\epsilon(v_2) = 1$, so $\epsilon(v_1) = (b-1) - \epsilon(v_2) = b-2$.
		\end{claimproof}
		
		\begin{claim}\label{claim:Kk}
			Every terminal set distinct from $B$ induces a clique of size $k$.
		\end{claim}
		\begin{claimproof}
			Suppose that $B'$ is a terminal set distinct from $B$ and $b'\defeq |B'| \leq k-1$. By  Claim~\ref{claim:deg2}, $B'$ is dense.
			Thus, by Lemma~\ref{lemma:S}, $|S_{B'}| \geq k - b'$, i.e., $S$ contains  at least $k-b'$ vertices  that are adjacent to every vertex in~$B'$. Consider $v \in S \setminus \set{v_1}$. By definition, $v$ has $b$ neighbors in~$B$; due to Claim~\ref{claim:S-v1}, $v$ also has $s - 2 = (k - b + 1) - 2 = k - b - 1$ neighbors in $S$. On the other hand, by Claim~\ref{claim:epsilon}, $\deg(v) = k+1$. Therefore, $$\deg_{D \setminus B}(v) = (k+1) - b - (k - b - 1) = 2.$$ In particular, $v$ cannot be adjacent to all the vertices in~$B'$. 
			Thus, $S_{B'} = \set{v_1}$ and $|B'| = k-1$. But
			$$\deg_{D\setminus B} (v_1) = \epsilon(v_1) + k - \deg_B(v_1) - \deg_S(v_1) =  (b-2) + k - b - 0 = k-2 < k-1;$$
			a contradiction.
		\end{claimproof}
		
		\begin{claim}\label{claim:two}
			There are exactly two terminal sets distinct from $B$.
		\end{claim}
		\begin{claimproof} Suppose $D \setminus B$ contains $\ell$ terminal sets.
			By Claim~\ref{claim:Kk}, the number of edges between $S$ and the terminal vertices of any terminal set
			$B'$ distinct from $B$ is at least $k-1$ and at most~$k$. On the other hand, the number of edges
			between $S$ and $D \setminus B$ is exactly $(k-2) + 2(k-b) = 3k - 2b - 2$.
			Therefore,  $$\ell(k-1)\leq 3k-2b-2 \leq \ell k,$$ so $1\leq \ell \leq 2$. However, if $\ell = 1$, then $3k-2b-2 \leq k$, so $b \geq k-1$, which contradicts Claim~\ref{claim:at_least_3}. Thus, $\ell = 2$, as desired.
		\end{claimproof}
		
		Now we are ready to finish the argument. Let $B_1$ and $B_2$ denote the only two terminal sets in $D\setminus B$, which, by Claim~\ref{claim:Kk}, induce cliques of size~$k$. We have $D\setminus B = C_{B_1} \cup C_{B_2}$. Notice that $v_1$ is adjacent to at least one terminal vertex in $B_1 \cup B_2$. Indeed, there are at least $2(k-1)$ edges between $S$ and the terminal vertices in $B_1 \cup B_2$, while each vertex in $S\setminus \set{v_1}$ has $2$ neighbors in $D \setminus B$, providing in total only $2(k-b)$ edges.
		
		Without loss of generality, assume that $v_1$ is adjacent to at least one terminal vertex in~$B_1$. Since $v_1$ has only $k-2$ neighbors in $D\setminus B$, Lemma~\ref{lemma:many_neighbors_outside} implies that
		$v_1$ has at least $k-1$ neighbors outside of $C_{B_1}$. Since $v_1$ has only $b \leq k-2$ neighbors outside of $C_{B_1} \cup C_{B_2}$, we see that $C_{B_1} \neq C_{B_2}$ and $v_1$ has a neighbor in $C_{B_2}$. Since $B_1$ and $B_2$ are the unique terminal sets in $C_{B_1}$ and $C_{B_2}$ respectively, we have $B_1 = C_{B_1}$ and $B_2 = C_{B_2}$. Therefore, $v_1$ is also adjacent to at least one terminal vertex in $B_2$ and, hence, has at least $k-1$ neighbors outside of $B_2$.
		
		Notice that $2k=|E_G(B_1 \cup B_2,S)|=3k-2b-2$, i.e., $k = 2b+2$. Let $d_i\defeq\deg_{B_i}(v_1)$. Then for each $i \in \set{1,2}$, $d_i\geq k-1-b$. Since $$b+d_1+d_2=\deg(v_1) = k+b-2,$$ we obtain that 
		$k+b-2\geq b+2(k-1-b)$, i.e., $2b\geq k$, contradicting $k=2b+2$.
	\end{proof}
	
	\begin{corl}\label{corl:S}
		Let $B$ be a dense terminal set. Then $G[B \cup S_B]$ is a clique of size $k$.
	\end{corl}
	\begin{proof}
		By Lemma~\ref{lemma:S}, $|B\cup S_B| \geq k$; on the other hand, by Lemma~\ref{lemma:S_is_a_clique}, $G[B \cup S_B]$ is a clique, so $|B \cup S_B| \leq k$.
	\end{proof}
	
	\begin{corl}\label{corl:2}
		There does not exist a subset $U \subseteq V$ of size $k+1$ such that $G[U]$ is a complete graph minus an edge with the two nonadjacent vertices in $\c{D}$.
	\end{corl}
	\begin{proof}
		Suppose, towards a contradiction, that $U$ is such a set and let $v_1$, $v_2 \in U \cap \c{D}$ be the two nonadjacent vertices in $U$. Set $B \defeq U \cap D$. Note that $|B| \geq |U| - |\c{D}| \geq (k+1) - (k-2) = 3$.
		
		Since for each $u \in B$, $\deg_U(u) = k$, there are no edges between $B$ and $\c{U}$. In particular, $B= C_B$. If $|B| \geq 4$, then $B$ is a dense set and $U = B \cup S_B$, which is impossible due to Corollary~\ref{corl:S}. Therefore, $|B| = 3$. Thus, $|U \setminus B| = (k+1) - 3 = k-2$, so $\c{D} = U \setminus B$. By Lemma~\ref{lemma:no_small_degrees}\ref{item:no_small_degrees:structure}, $G[\c{D}]$ is a disjoint union of cliques. On the other hand, $G[\c{D}]$ is a complete graph minus the edge $v_1 v_2$. The only possibility then is that $|\c{D}| = 2$, i.e., $k = 4$. Each vertex in $\c{D}$ is of degree $5$ and, therefore, has exactly $2$ neighbors in $D \setminus B$. By Lemma~\ref{lemma:not_all_sparse}, there exists a dense terminal set $B' \subseteq D \setminus B$. Since $k = 4$, we must have $|B'| = 4$, so there are $4$ edges between $B'$ and $\c{D}$. This implies that $D \setminus B = B'$ and each vertex in $\c{D}$ has exactly $2$ neighbors in $B'$. But then $G\in \D_4$.
	\end{proof}
	
	\subsection{The graph $G[T_B]$}
	
	In this section we show that if $B$ is a dense terminal set, then $G[T_B]$ is a clique. However, in order for some of our arguments to go through, we need to establish some of the results for the more general case when $B$ is a terminal set such that $G[B \cup S_B]$ is a clique of size $k$ (i.e., $G[B]$ can also be isomorphic to a triangle).
	
	\begin{lemma}\label{lemma:small_deg}
		Let $B$ be a terminal set such that $G[B \cup S_B]$ is a clique of size $k$. Then every vertex in $S_B$ has at most $|B|-1$ neighbors outside of $B \cup S_B$.
	\end{lemma}
	\begin{proof}
		Set $S \defeq S_B$. Let $v \in S$ and suppose that $v$ has $d$ neighbors outside of $B \cup S$. Then
		$$
		\epsilon(v) = \deg_{B \cup S}(v) + \deg_{\c{(B \cup S)}}(v) - k = (k-1) + d - k = d-1,
		$$
		so, using that $|S| = k - |B|$, we obtain
		$$
		k-2 \geq \epsilon(\c{D}) = \epsilon(S) + \epsilon(\c{D} \setminus S) \geq (d-1) + (k - |B| -1) + |\c{D}\setminus S|,
		$$
		i.e., $d \leq |B|-|\c{D}\setminus S|$. It remains to notice that $\c{D} \setminus S \neq \0$, since each terminal vertex in $B$ has a neighbor in $\c{D}\setminus S$.
	\end{proof}
	
	\begin{lemma}\label{lemma:structure_of_lists}
		Let $B$ be a terminal set such that $G[B \cup S_B]$ is a clique of size $k$. Let $I$ be a coloring with $\dom(I) = \c{(B \cup S_B)}$. Then for any $u \in B$, $|L_I(u)| = k-1$, and for any two distinct $u_1$, $u_2 \in B$, the matching $E_H(L_I(u_1), L_I(u_2))$ is perfect.
	\end{lemma}
	\begin{proof}
		Set $S \defeq S_B$. By Lemma~\ref{lemma:small_deg}, $|L_I(v)| \geq k-|B|+1$ for all $v \in S$. Since $|S| = k - |B|$, $I$ can be extended to a coloring $I'$ with $\dom(I') = \c{B}$. Therefore, due to Lemma~\ref{lemma:spoilt}\ref{item:spoilt:problem} and since $G[B]$ is connected, $I$ does not enhance any $u \in B$, i.e., $|L_I(u)| = k-1$, as claimed. Now, let $u_1$, $u_2$ be two distinct vertices in~$B$ and suppose, towards a contradiction, that $x \in L_I(u_1)$ has no neighbor in $L_I(u_2)$. For each $v \in S$, let $L'(v) \defeq L_I(v) \setminus N_H(x)$. Then $|L'(v)| \geq k-|B| = |S|$ for all $v \in S$, so there is $I' \in \I(H)$ with $\dom(I') = S$ such that $I' \subseteq \bigcup_{v \in S} L'(v)$. Then $I \cup I'$ is a coloring with domain $\c{B}$; moreover, $x \in L_{I \cup I'}(u_1)$, which implies that the matching $E_H(L_{I \cup I'}(u_1), L_{I \cup I'}(u_2))$ is not perfect. Due to Lemma~\ref{lemma:edge_extends}, $I \cup I'$ can be extended to an $\Cov{H}$-coloring of $G$; a contradiction. 
	\end{proof}
	
	\begin{lemma}\label{lemma:T}
		Let $B$ be a terminal set such that $G[B \cup S_B]$ is a clique of size $k$. Then $G[T_B]$ is a clique of size at least $2$.
	\end{lemma}
	\begin{proof}
		Set $S \defeq S_B$ and $T \defeq T_B$. First, observe that $|T| \geq 2$: Each vertex in $B$ has a (unique) neighbor in $T$; thus, if $|T| = 1$, then the only vertex in $T$ has to be adjacent to all the vertices in $B$, which contradicts the way $T$ is defined.
		
		Now suppose that $v_1$, $v_2 \in T$ are two distinct nonadjacent vertices. For each $i \in \set{1, 2}$, choose a neighbor $u_i \in B$ of $v_i$.  Since every vertex in $B$ has only one neighbor outside of $B \cup S$, $u_1 v_2$, $u_2 v_1 \not \in E$. Note that, by Lemma~\ref{lemma:structure_of_lists}, there are at least $k-1$ edges between $L(u_1)$ and $L(u_2)$. Let $H'$ be the graph obtained from $H$ by adding, if necessary, a single edge between $L(u_1)$ and $L(u_2)$ that completes a perfect matching between those two sets. Let $H^\ast$ be the graph obtained from $H$ by adding a matching $M$ between $L(v_1)$ and $L(v_2)$ in which $x_1 \in L(v_1)$ is adjacent to $x_2 \in L(v_2)$ if and only if there exist $y_1 \in L(u_1)$, $y_2 \in L(u_2)$ such that $x_1 y_1 y_2 x_2$ is a path in $H'$. Observe that $\Cov{H}^\ast \defeq (L,H^\ast)$ is a cover of the graph $G^\ast$ obtained from $G$ by adding the edge $v_1v_2$.
		
		\begin{claim}
			There is  no independent set $I \in \I(H^\ast)$ with $\dom(I) = \c{(B \cup S)}$.
		\end{claim}
		\begin{claimproof}
			Assume, towards a contradiction, that $I \in \I(H^\ast)$ is such that $\dom(I) = \c{(B \cup S)}$. Since, in particular, $I \in \I(H)$, Lemma~\ref{lemma:structure_of_lists} guarantees that the edges of $H$ between $L_I(u_1)$ and $L_I(u_2)$ form a perfect matching of size $k-1$. For each $i \in \set{1, 2}$, let $y_i$ be the unique element of $L(u_i) \setminus L_I(u_i)$. Then $y_1y_2$ is an edge in $H'$. However, since $y_i \not \in L_I(u_i)$, the unique element of $I \cap L(v_i)$, which we denote by $x_i$, is adjacent to $y_i$ in $H$. Therefore, $x_1 y_1 y_2 x_2$ is a path in $H'$, so $x_1x_2$ is an edge in $H^\ast$. This contradicts the independence of $I$ in $H^\ast$.
		\end{claimproof}
		
		Let $W \subseteq \c{(B \cup S)}$ be an inclusion-minimal subset for which there is no $I \in \I(H^\ast)$ with $\dom(I) = W$. 
		Since $G$ is $\Cov{H}$-critical, $G^\ast[W]$ is not a subgraph of $G$, so $\set{v_1, v_2} \subseteq W$. Since for all $v \in W$, $\deg(v) \geq \deg_{G^\ast[W]}(v)$, we have
		$$
		\epsilon(W) \geq \sum_{v \in W} (\deg_{G^\ast[W]}(v) - k).
		$$
		In particular,
		$$
		\sum_{v \in W} (\deg_{G^\ast[W]}(v) - k) \leq k-2.
		$$
		By the minimality of $G$, either $G^\ast[W] \in \D_k$, or else, $G^\ast[W]$ contains a clique of size $k+1$.
		
		If $G^\ast[W] \in \D_k$, then
		$$
			\sum_{v \in W} (\deg_{G^\ast[W]}(v) - k) = k-2.
		$$
		Therefore, $\c{D} \subseteq W$ and $\deg(v) = \deg_{G^\ast[W]}(v)$ for all $v \in W$. The latter condition implies that the only vertices in $W$ that are adjacent to a vertex in $B$ are $v_1$ and $v_2$; moreover, the only neighbor of $v_1$ in $B$ is $u_1$ and the only neighbor of $v_2$ in $B$ is $u_2$. Since $\c{D} \subseteq W$, this implies $S = \0$ and $T \cap \c{D} \subseteq \set{v_1, v_2}$. Therefore, $|B| = k - |S| = k$. Each terminal vertex in $B$ has a neighbor in $T \cap \c{D}$, so the set of all terminal vertices in $B$ is a subset of $\set{u_1, u_2}$. Since $k \geq 3$, this implies $k = |B| = 3$ and $u_1$, $u_2$ are indeed the terminal vertices in $B$. But then every vertex in $\c{D}$ is adjacent both to $u_1$ and to $u_2$, contradicting the fact that $v_1$ and $v_2$ each have only one neighbor among $u_1$, $u_2$.
		
		Thus, $G^\ast[W]$ contains a clique of size $k+1$. Since $G$ does not contain such a clique, there exists a set $U \subseteq \c{(B \cup S \cup \set{v_1, v_2})}$ of size $k-1$ such that the graph $G[U \cup \set{v_1, v_2}]$ is isomorphic to $K_{k+1}$ minus the edge $v_1v_2$. Note that $U\not\subseteq\c{D}$, since $|\c{D}| \leq k-2$. Thus, the set $B' \defeq U \cap D$ is nonempty. Let $S' \defeq U \setminus B'$. Each vertex in $B'$ has $k$ neighbors in $U \cup \set{v_1, v_2}$, so there are no edges between $B'$ and $\c{(U \cup \set{v_1, v_2})}$. Due to Corollary~\ref{corl:2}, $\set{v_1, v_2} \not \subseteq \c{D}$, so we can assume, without loss of generality, that $v_2 \in D$ and let $B^\ast \defeq B' \cup \set{v_2}$. Then $G[B \cup B^\ast]$ is a connected component of $G[D]$, with $u_2v_2$ being a unique edge between terminal sets $B$ and $B^\ast$. Note that $S' = S_{B^\ast}$ and $G[B^\ast \cup S']$ is a clique of size $k$. Moreover, $v_1u_2 \not \in
		E$ and $\set{v_1, u_2} \subseteq T_{B^\ast}$. Thus, we can apply the above reasoning to $B^\ast$ in place of $B$ and $v_1$, $u_2$ in place of $v_1$, $v_2$. As a result, we see that $G[B \cup S \cup \set{v_1}]$ is isomorphic to $K_{k+1}$ minus the edge $v_1u_2$. Therefore,
		$$
		\epsilon(v_1) \geq \deg_{B \cup S}(v_1) + \deg_{B^\ast \cup S'}(v_1) - k = (k-1)+(k-1)-k = k-2.
		$$
		Thus, $\c{D} = \set{v_1}$, $S = S' = \0$, and $|B| = |B^\ast| = k$. This implies that $G \in \D_k$.
	\end{proof}	
	
	\begin{corl}\label{corl:T}
		Let $B$ be a dense terminal set. Then $G[T_B]$ is a clique of size at least $2$.
	\end{corl}
	\begin{proof}
		Follows from Corollary~\ref{corl:S} and Lemma~\ref{lemma:T}.
	\end{proof}
	
	\subsection{The graph $G[S_B \cup T_B]$}
	
	\begin{lemma}\label{lemma:|T|_is_2}
		Let $B$ be a dense terminal set. Then:
		\begin{enumerate}[wide,label=\normalfont{(\roman*)}]
			\item \label{item:|T|_is_2:2} $|T_B| = 2$;
			\item \label{item:|T|_is_2:D_complement} $\c{D} = S_B \cup (T_B \cap \c{D})$;
			\item \label{item:|T|_is_2:k-1} each vertex in $T_B$ has exactly $k-1$ neighbors outside of $B$; and
			\item \label{item:|T|_is_2:epsilon} $\epsilon(v) = 1$ for all $v \in S_B$.
		\end{enumerate}
	\end{lemma}
	\begin{proof}
		Let $S \defeq S_B$ and $T \defeq T_B$. By Corollaries~\ref{corl:S} and~\ref{corl:T}, $G[S\cup B]$ is a clique of size $k$ and $G[T]$ is a clique of size at least $2$.
		
		Suppose that \ref{item:|T|_is_2:2} does not hold, i.e., $|T| \geq 3$. Recall that, by Lemma~\ref{lemma:many_neighbors_outside}, each vertex in $T$ has at least $k-1$ neighbors outside of $B$. If $T$ contains  at most one vertex with exactly $k-1$ neighbors outside of $B$, then
		$$
		\epsilon(S) + \epsilon(T) \geq |S| + \sum_{v \in T} \deg(v) - k|T| \geq (k-|B|) + (|B| + k|T| - 1) - k|T| = k-1;
		$$
		a contradiction. Thus, there exist two distinct vertices $v_1$, $v_2 \in T$ such that $$\deg_{\c{B}}(v_1) = \deg_{\c{B}}(v_2) = k-1.$$ Since $|T| \geq 3$ and every vertex in $B$ has exactly one neighbor in $T$, there exists a vertex $u_0 \in B$ such that $u_0v_1$, $u_0v_2 \not \in E$. Also, we can choose a vertex $u_1 \in B$ with $u_1v_1 \in E$; note that $u_1v_2 \not \in E$. Let $I \in \I(H)$ be such that $\dom(I) = \c{(B \cup \set{v_1, v_2})}$. Then
		$$\phi_{B \cup \set{v_1, v_2}}(v_1) = \phi_{B \cup \set{v_1, v_2}}(v_2) = k - (k-2) = 2.$$
		(Here we use that $v_1$ and $v_2$ are adjacent to each other.) Let $x_1$, $x_2$ be any two distinct elements of $L_I(v_1)$ and choose $y_1$, $y_2 \in L_I(v_2)$ so that $x_1y_1$, $x_2y_2 \not \in E(H)$.  Since $$L_{I \cup \set{x_1, y_1}}(u_0) = L_{I\cup \set{x_2, y_2}}(u_0) = L_I(u_0),$$ and for each $i \in \set{1,2}$, $$L_{I \cup\set{x_i, y_i}}(u_1) = L_{I \cup\set{x_i}}(u_1),$$ Lemma~\ref{lemma:edge_extends} implies that for each $i\in \set{1,2}$, the matching $E_H(L_I(u_0),  L_{I \cup\set{x_i}}(u_1))$ is perfect. But then the unique vertex in $L_I(u_1)$ that has no neighbor in $L_I(u_0)$ is adjacent to both $x_1$ and $x_2$, which is impossible. This contradiction proves~\ref{item:|T|_is_2:2}.
		
		In view of~\ref{item:|T|_is_2:2}, we now have
		\begin{equation}\label{ej6}
		\epsilon(\c{D}) \geq \epsilon(S) + \epsilon(T) \geq (k - |B|) + (|B| + 2(k-1)) - 2k = k-2,
		\end{equation}
		so none of the  inequalities in~\eqref{ej6} can be strict. This yields \ref{item:|T|_is_2:D_complement}, \ref{item:|T|_is_2:k-1}, and \ref{item:|T|_is_2:epsilon}.
	\end{proof}
	
	\begin{lemma}\label{lemma:entire}
		Let $B$ be a dense terminal set. Then $B = C_B$.
	\end{lemma}
	\begin{proof}
		Suppose, towards a contradiction, that $B \neq C_B$. Then $T_B \cap D \neq \0$. On the other hand, every terminal vertex in $B$ has a neighbor in $T_B \cap \c{D}$, so we also have $T_B \cap \c{D} \neq \0$. By Lemma~\ref{lemma:|T|_is_2}\ref{item:|T|_is_2:2}, $|T_B| = 2$, so $T_B \eqqcolon \set{v, u}$, where $v \in \c{D}$ and $u \in D$, with $v$ adjacent to all the terminal vertices in $B$. By Lemma~\ref{lemma:|T|_is_2}\ref{item:|T|_is_2:D_complement}, $\c{D} = S_B \cup \set{v}$.
		By Corollary~\ref{corl:S}, $G[B \cup S_B] \cong K_k$; in particular, $|S_B| = k - |B|$. Therefore,
		\begin{equation}\label{eq:cD}
			|\c{D}| =  k-|B|+1.
		\end{equation}
		Let $B'$ be any other terminal set such that $C_{B'} = C_B$. Note that $B'$ is dense, since, otherwise, by Lemma~\ref{lemma:no_small_degrees}\ref{item:no_small_degrees:bounds}, $|\c{D}| = k - 2$, contradicting~\eqref{eq:cD}. Therefore, the above reasoning can be applied to $B'$ in place of $B$. In particular, $T_{B'} \eqqcolon \set{v', u'}$, where $v' \in \c{D}$ and $u' \in D$, with $v'$ adjacent to all the terminal vertices in $B'$. Moreover, $\c{D} = S_B \cup \set{v} = S_{B'} \cup \set{v'}$. Consider any vertex $w \in \c{D}$. If $w \in S_{B'}$, then, by definition, $w$ is adjacent to every vertex in $B'$. If, on the other hand, $w = v'$, then $w$ is adjacent to all the terminal vertices in $B'$. In either case, $w$ has at least $|B'|-1$ neighbors in $B'$. However, if $w \in S_B$, then due to Lemma~\ref{lemma:|T|_is_2}\ref{item:|T|_is_2:epsilon}, $w$ has exactly $(k+1) - (k-1) = 2$ neighbors outside of $B \cup S_B$. This implies that $S_B = \0$, and similarly $S_{B'} = \0$. Thus, $v = v'$ and $|B| = |B'| = k$. Then $v$ is adjacent to $k-1$ terminal vertices in $B'$ and to $u$, contradicting Lemma~\ref{lemma:|T|_is_2}\ref{item:|T|_is_2:k-1}.
	\end{proof}
	
	\begin{corl}\label{corl:D^c}
		Let $B$ be a dense terminal set. Then $\c{D} = S_B \cup T_B$.
	\end{corl}
	\begin{proof}
		Follows immediately by Lemma~\ref{lemma:entire} and Lemma~\ref{lemma:|T|_is_2}\ref{item:|T|_is_2:D_complement}.
	\end{proof}
	
	\subsection{Finishing the proof of Theorem~\ref{theo:sharp_DP_Dirac}}
	
	\begin{lemma}\label{lemma:exists_sparse}
		There exists a sparse terminal set.
	\end{lemma}
	\begin{proof}
		Suppose, towards a contradiction, that every terminal set is dense. Lemma~\ref{lemma:entire} implies that in such case every connected component of $G[D]$ is a clique of size at least $4$. Moreover, due to Corollary~\ref{corl:D^c}, Corollary~\ref{corl:S}, and Lemma~\ref{lemma:|T|_is_2}\ref{item:|T|_is_2:2}, the size of every connected component of $G[D]$ is precisely $k - |\c{D}| + 2 \eqqcolon b$. Note that due to Lemma~\ref{lemma:|T|_is_2}\ref{item:|T|_is_2:k-1}, the graph $G[D]$ is disconnected.
		
		Let $B_1$ and $B_2$ be the vertex sets of any two distinct connected components of $G[D]$. Lemma~\ref{lemma:|T|_is_2}\ref{item:|T|_is_2:epsilon} implies that $S_{B_1} \cap S_{B_2} = \0$, since every vertex in $S_{B_1}$ has only $2$ neighbors outside of $B_1 \cup S_{B_1}$. Since, by Corollary~\ref{corl:D^c},
		$$\c{D} = S_{B_1} \cup T_{B_1} = S_{B_2}\cup T_{B_2},$$ it follows that $S_{B_1} \subseteq T_{B_2}$ and $S_{B_2} \subseteq T_{B_1}$. Therefore, $|S_{B_1}| \leq |T_{B_2}|$, i.e., $k-b \leq 2$, which implies
		$$b \in \set{k-2, k-1, k}.$$ Now it remains to consider the three possibilities.
		
		{\sc Case~1:} $b = k-2$. 
		Let $B$ be the vertex set of any connected component of $G[D]$. Set $T_B \eqqcolon \set{v_1, v_2}$ and let $u_1$, $u_2 \in B$ be such that $u_1 v_1$, $u_2 v_2 \in E$. Choose any $x \in L(v_1)$. Note that
		$|L_{\set{x}}(v_2)| \geq 2$, so we can choose $y \in L_{\set{x}}(v_2)$ in such a way that $E_H(L_{\set{x, y}}(u_1), L_{\set{x, y}}(u_2))$ is not a perfect matching. For all $u \in D \setminus B$, we have $|L_{\set{x,y}}(u)| \geq k-2$ and the size of every connected component of $G[D\setminus B]$ is $k-2$. Therefore, there exists a coloring $I$ with $\dom(I) = D \setminus B$ such that $I \cup \set{x, y} \in \I(H)$. But then
		$$\dom(I \cup \set{x,y}) = \c{(B \cup S_B)},$$ and the matching $E_H(L_{I\cup\set{x,y}}(u_1), L_{I\cup\set{x,y}}(u_1))$ is not perfect, contradicting Lemma~\ref{lemma:structure_of_lists}.
		
		{\sc Case~2:} $b = k-1$. Let $B_1$ and $B_2$ be the vertex sets of any two distinct connected components of $G[D]$. Let $v$ be the unique vertex in $S_{B_1}$. Then $v \in T_{B_2}$ (recall that $S_{B_1} \cap S_{B_2} = \0$), so, by Lemma~\ref{lemma:|T|_is_2}\ref{item:|T|_is_2:k-1}, $v$ has exactly $k-1$ neighbors outside of $B_2$. By Corollary~\ref{corl:T}, one of the neighbors of $v$ is the other vertex in $T_{B_2}$. Therefore, $v$ can have at most $k-2$ neighbors in $B_1$; a contradiction with the choice of $v$.
		
		{\sc Case~3:} $b = k$. In this case, $G[\c{D}] \cong K_2$ and every vertex in $D$ has exactly one neighbor in~$\c{D}$, so there are exactly $k$ edges between $\c{D}$ and every connected component of $G[D]$. On the other hand, if $B$ is the vertex set of a connected component of $G[D]$, then, by Lemma~\ref{lemma:|T|_is_2}\ref{item:|T|_is_2:k-1}, there are exactly $2(k-2) < 2k$ edges between $\c{D}$ and $D \setminus B$. Thus, the graph $G[D\setminus B]$ is connected. Moreover, $k \geq 4$, for $2\cdot (3-2) = 2 < 3$. Let $B' \defeq D \setminus B$ (so $G[B']$ is a clique of size $k$). Set $\c{D} \eqqcolon \set{v_1, v_2}$ an let $u_1$, $u_2 \in B$, $u_1'$, $u_2' \in B'$ be such that $u_1 v_1$, $u_2 v_2$, $u_1' v_1$, $u_2' v_2 \in E$. Choose any $x \in L(v_1)$. Note that
		$|L_{\set{x}}(v_2)|\geq 3$. There is at most one element $y \in L_{\set{x}}(v_2)$ such that $E_H(L_{\set{x, y}}(u_1), L_{\set{x,y}}(u_2))$ is a perfect matching; similarly for $u_1'$ and $u_2'$. Therefore, there exists $z \in L_{\set{x}}(v_2)$ such that neither $E_H(L_{\set{x, z}}(u_1), L_{\set{x,z}}(u_2))$ nor
		$E_H(L_{\set{x, z}}(u_1'), L_{\set{x,z}}(u_2'))$ are perfect matchings. Thus, by Lemma~\ref{lemma:edge_extends}, $\set{x,z}$ can be extended to an $\Cov{H}$-coloring of $G$; a contradiction.
	\end{proof}
	
	Now we are ready to finish the proof of Theorem~\ref{theo:sharp_DP_Dirac}. Let $B$ be a dense terminal set (which exists by Lemma~\ref{lemma:not_all_sparse}) and let $B'$ be a sparse terminal set (which exists by Lemma~\ref{lemma:exists_sparse}). By Lemma~\ref{lemma:entire}, $B = C_B$ and every terminal set in $C_{B'}$ is sparse. In particular, $G[C_{B'}]$ contains at least $3$ vertices of degree $2$. Thus, by Lemma~\ref{lemma:no_small_degrees}\ref{item:no_small_degrees:bounds}, every vertex in $\c{D}$ has at least $3$ neighbors in $C_{B'}$. On the other hand, by Lemma~\ref{lemma:|T|_is_2}\ref{item:|T|_is_2:epsilon}, a vertex in $S_B$ has only $2$ neighbors in $\c{(B \cup S_B)}$. Therefore, $S_B = \0$. Due to Corollary~\ref{corl:D^c}, we obtain $\c{D} = T_B$; thus, by Corollary~\ref{corl:T} and Lemma~\ref{lemma:|T|_is_2}\ref{item:|T|_is_2:2}, $G[\c{D}] \cong K_2$. On the other hand, by Lemma~\ref{lemma:no_small_degrees}\ref{item:no_small_degrees:bounds}, $|\c{D}| = k-2$, so $k = 4$. But each vertex in $T_B$ has at least $4$ neighbors outside of $B$ ($1$ in $T_B$ by Corollary~\ref{corl:T} and $3$ in $C_{B'}$ by Lemma~\ref{lemma:no_small_degrees}\ref{item:no_small_degrees:bounds}), which contradicts Lemma~\ref{lemma:|T|_is_2}\ref{item:|T|_is_2:k-1}.
	
	\section{Concluding remarks}
	
	In~\cite{BKP}, the notion of DP-coloring was naturally extended to multigraphs (with no loops). The only difference from the graph case is that if distinct vertices $u$, $v \in V(G)$ are connected by $t$ edges in $G$, then the set $E_H(L(u), L(v))$ is a union of $t$ matchings (not necessarily perfect and possibly empty). Bounding the difference $2|E(G)| - k|V(G)|$ for DP-critical multigraphs~$G$ appears to be a challenging problem.
	
	\begin{defn}
		For $k \geq 3$, a \emph{$k$-brick} is a $k$-regular multigraph  whose underlying simple graph is either a clique or a cycle and
		in which the multiplicities of all edges are the same.
	\end{defn}
	
	Note that for a $k$-brick $G$, $2|E(G)| = k|V(G)|$. In~\cite{BKP}, it is shown that $k$-bricks are the only $k$-DP-critical multigraphs with this property.
	
	Theorem~\ref{theo:sharp_DP_Dirac} fails for multigraphs, as the following example demonstrates. Fix an integer $k \in \N$ divisible by $3$ and let $G$ be the multigraph with vertex set $[3]$ such that $|E_G(1, 2)|=k/3$ and $|E_G(1, 3)| = |E_G(2, 3)| = 2k/3$, so we have $2|E(G)| - k|V(G)| = k/3$. Let $H$ be the graph with vertex set $[3] \times [3] \times [k/3]$ in which two distinct vertices $(i_1, j_1, a_1)$ and $(i_2, j_2, a_2)$ are adjacent if and only if one of the following three (mutually exclusive) situations occurs:
	\begin{enumerate}
		\item $\set{i_1, i_2} = [2]$ and $j_1 = j_2$;
		\item $\set{i_1, i_2} \neq [2]$ and $j_1 \neq j_2$; or
		\item $(i_1, j_1) = (i_2, j_2)$.
	\end{enumerate}
	For each $i \in [3]$, let $L(i) \defeq \set{i} \times [3] \times [k/3]$. Then $\Cov{H} \defeq (L, H)$ is a $k$-fold cover of $G$. We claim that $G$ is not $\Cov{H}$-colorable. Indeed, suppose that $I \in \I(H)$ is an $\Cov{H}$-coloring of $G$ and for each $i \in [3]$, let $I \cap L(i) \eqqcolon \set{(i, j_i, a_i)}$. By the definition of~$H$, we have $j_1 \neq j_2$, while also $j_1 = j_3 = j_2$, which is a contradiction. It is also easy to check that $G$ is $\Cov{H}$-critical and that it does not contain any $k$-brick as a subgraph.
	
	In light of the above example, we propose the following problem:
	
	\begin{problem} Let $k\geq 3$.
		Let $G$ be a multigraph and let $\Cov{H}$ be a $k$-fold cover of $G$ such that $G$ is $\Cov{H}$-critical. Suppose that $G$ does not contain any $k$-brick as a subgraph. What is the minimum possible value of the difference $2|E(G)| - k|V(G)|$, as a function of $k$?
	\end{problem}
	
	\paragraph*{Acknowledgement.} The authors are very grateful to the anonymous referee for the valuable comments and suggestions.
	
	\printbibliography
	
\end{document}